\documentclass[a4paper,11pt]{article}

\usepackage[left=2.5cm,right=2.5cm,top=3cm,bottom=3cm,pdftex]{geometry}
\usepackage{amssymb}
\usepackage{amsmath}
\usepackage{amsthm}
\usepackage{dsfont}
\usepackage{url}
\usepackage{natbib}

\usepackage[utf8]{inputenc}
\usepackage[T1]{fontenc}

\usepackage[pdftex]{graphicx}
\DeclareGraphicsExtensions{.png, .pdf, .jpg} % {.jpg,.pdf,.mps,.png,.gif}

\usepackage[pdftex, colorlinks, linkcolor=blue, urlcolor=blue, citecolor=blue, breaklinks=true]{hyperref}

\newtheorem{thm}{Theorem}
\newtheorem{lem}{Lemma}

\begin{document}

\begin{center}
\Large\textbf{Optimal Designs in Multiple Group Random Coefficient Regression Models} \\[11pt]
\normalsize
Maryna Prus\footnote{Maryna Prus: \href{mailto:maryna.prus@ovgu.de}{maryna.prus@ovgu.de}}\\[11pt]

\footnotesize
Otto-von-Guericke University Magdeburg, Institute for Mathematical Stochastics, 
\\
PF 4120, D-39016 Magdeburg, Germany\\
\normalsize
\end{center}

\begin{quote}
\textbf{Abstract:} The subject of this work is multiple group random coefficients regression models with several treatments and one control group. Such models are often used for studies with cluster randomized trials. We investigate \textit{A}-, \textit{D}- and \textit{E}-optimal designs for estimation and prediction of fixed and random treatment effects, respectively, and illustrate the obtained results by numerical examples.

\end{quote}

\section{Introduction}

Random coefficient regression (RCR) models were originally introduced in plant and animal breeding and used for selection purposes (see e.\,g. \cite{hen4}). 
The subject of this paper is the multiple group random coefficient models, in which observational units (individuals) are allocated in several treatment groups and one control group. In each treatment group some group-special kind of treatment is available, in the control group there is no treatment. Such models are typically used for cluster randomization or cluster randomized trials (see e.\,g. \cite{bla} or \cite{pat}). 

RCR models with known population parameters were considered in detail by \cite{gla}. 
\cite{ent} investigated optimal designs for \textit{estimation} of unknown population mean parameters in RCR models, where all individuals are observed under the same regime. Analytical results for designs, which are optimal for the \textit{prediction} of individual random parameters in hierarchical models with the same treatment for all individuals, have been presented in \cite{pru1}. A practical approach for computation of optimal approximate and exact designs was proposed by \cite{harm}.

The fixed effects version of the multiple group models considered in this paper may be recognized as the well known one way layout model. Classical one-way layout models have been well discussed in the literature. Results for the optimal designs can be found e.\,g. in \cite{bai}, \cite{rasch}, \cite{wie}, \cite{schw} or \cite{maju}.

Optimal designs for \textit{estimation} of fixed parameters in multiple group models with random coefficients were considered e.\,g. in \cite{fedo}, \cite{sch}, \cite{kun}, \cite{blu} and \cite{lem}. \cite{fedo} worked on optimal designs, which minimize a loss function in multicentre trials models. In \cite{sch} models with the same fixed number of observations in all groups were investigated. \cite{kun} proposed a design optimization method based on the generalized least squares estimation. \cite{blu} considered models with carryover effects. In \cite{lem} optimal designs were computed for the maximum likelihood estimation.

Optimal designs for \textit{prediction} of random effects in models, where the population mean parameters differ from group to group, were briefly discussed in \cite{pru3}. 
 
In this paper we investigate multiple group models with the same unknown population parameters across all groups. We present analytical results for \textit{A}-, \textit{D}- and \textit{E}-optimal designs (-optimal group sizes) based on the best linear unbiased \textit{estimation} or \textit{prediction} of fixed or random treatment effects, respectively.

The paper has the following structure: In Section~2 the model will be specified. In Section~3 the best linear unbiased \textit{estimation} for the population parameters (mean treatment effects) and a best linear unbiased \textit{prediction} for the random treatment effects of the observational units will be discussed. Section~4 provides analytical results for designs, which are optimal for \textit{estimation} or \textit{prediction}. The results will be illustrated by a numerical example, in which we compare the optimal group sizes in the model under investigation with optimal group sizes in the fixed effects model (one-way-layout). The paper will be concluded by a discussion of the obtained results and possible directions for the future research in Section~5.

\section{Model Specification}\label{k1}
We consider here a multiple group model with $J$ groups and $N$ individuals. In the first $J-1$ (treatment) groups individuals get group-special kinds of treatment: $1, \dots, J-1$. Each treatment group includes $n>0$ individuals. The last group (group $J$) is a control group (no treatment) with $m=N-(J-1)n$, $m>0$ individuals. 
%We assume the numbers of individuals in the treatment and control groups to be positive ($n,m>0$).
The $k$-th observation at individual $i$ in the $j$-th treatment group is given by the following formula:
\begin{equation}\label{1}
{Y}_{jik}=\mu_{i}+\alpha_{ji} + \varepsilon_{jik},\quad k=1, \dots, K,\quad i=(j-1)n+1, \dots, jn,\quad j=1, \dots, J-1,
\end{equation}
while in the control group the $k$-th observation at the $i$-th individual is given by
\begin{equation}\label{2}
{Y}_{Jik}=\mu_{i} + \varepsilon_{Jik},\quad k=1, \dots, K,\quad i=(J-1)n+1, \dots, N,
\end{equation}
where $K$ is the number of observations per individual, which is assumed to be the same for all individuals across all groups, $\varepsilon_{jik}$ are observational errors with zero expected value and common variance $\sigma^2>0$.

In this work we optimize the numbers of individuals in treatment and control groups. Therefore the group allocation of individuals is not completely clear. For this reason we define the individual treatment effects $\alpha_{ji}$ for all individuals and all treatments, i.\,e. $i=1, \dots, N$ and $j=1, \dots, J-1$. For individual $i$ in the control group the treatment effects $\alpha_{ji}$ would appear if the individual would be treated with treatment $j$. For individual $i$ in treatment group $j$ the treatment effect $\alpha_{j'i}$ would appear if the individual would get treatment $j'$ instead of $j$ for $j, j'=1, \dots, J-1$, $j\neq j'$.

Individual intercepts $\mu_{i}$ are defined for all individual, i.\,e. $i=1, \dots, N$. $\mu_{i}$ and $\alpha_{ji}$ have unknown expected values $\mathrm{E}(\mu_{i})=\mu$, $\mathrm{E}(\alpha_{ji})=\alpha_j$ and variances $\mathrm{var}(\mu_{i})=u\,\sigma^2$, $\mathrm{var}(\alpha_{ji})=v\,\sigma^2$ for some given positive values of $u$ and $v$. All individual parameters $\mu_{i}$ and $\alpha_{ji'}$ and all observational errors $\varepsilon_{j'i''k}$, for $k=1, \dots, K$, $i,\,i'=1, \dots, N$ and $j, j', j''=1, \dots, J$, are uncorrelated.

%In this work we search for the optimal treatment group size\, $n$\, for estimation of unknown fixed effects\, $\alpha_j$\, or prediction of random effects\, $\alpha_{ji}$\,. Since the group size is to be optimized, the group allocation is not clearly for individuals.
%\begin{equation}\label{thetai}
%\mbox{\boldmath{$\theta$}}_i:= \left(\begin{array}{c} \mu_{i} \\ \alpha_{1\,i} \\ \vdots \\ \alpha_{J-1\,i} \end{array}\right)
%\end{equation}
%\begin{equation}
%\mathrm{E}(\mbox{\boldmath{$\theta$}}_i)=\left( \begin{array}{c} \mu \\ \alpha_{1} \\ \vdots \\ \alpha_{J-1} \end{array} \right)
%\end{equation}
%\begin{equation}
%\mathrm{Cov}(\mbox{\boldmath{$\theta$}}_i)=\sigma^2\left( \begin{array}{cc} u & 0 \\ 0 & \, v\, \mathbb{I}_{J-1} \end{array} \right),
%\end{equation}
%We denote the expected value of\, $\mbox{\boldmath{$\theta$}}_i$\, by
%\begin{equation}\label{fix}
%\mbox{\boldmath{$\theta$}}_0:=\mathrm{E}(\mbox{\boldmath{$\theta$}}_i)
%\end{equation}
%and the dispersion matrix (covariance matrix divided by\, $\sigma^2$\,) by
%\begin{equation}
%\mathbf{D}:=\left( \begin{array}{cc} u & 0 \\ 0 & \, v\, \mathbb{I}_{J-1} \end{array} \right),
%\end{equation}
%so that\, $\mathrm{Cov}(\mbox{\boldmath{$\theta$}}_i)=\sigma^2\mathbf{D}$.

For further considerations we define the regression functions
\begin{equation*}
\mathbf{f}(j):=\left\{\begin{array}{ll} (1,\mathbf{e}_j^\top)^\top \,, & j=1, \dots, J-1 \\ (1,\mathbf{0}_{J-1}^\top)^\top\,, & j=J \end{array}\right.,
\end{equation*}
where $\mathbf{e}_j$ and $\mathbf{0}_{J-1}$ denote the $j$-th unit vector and the zero vector of length $J-1$, respectively, and the total number of individuals
\begin{equation*}
N_j:=\left\{\begin{array}{ll} nj, & j=1, \dots, J-1 \\ N, & j=J \end{array}\right.
\end{equation*}
in the first $j$ groups, and we fix $N_0$ by $0$. Then for the vector $\mbox{\boldmath{$\theta$}}_i:= (\mu_{i}, \alpha_{1\,i}, \dots, \alpha_{J-1\,i})^\top$ of individual random parameters the multiple group model (defined by \eqref{1} and \eqref{2}) may be rewritten in the following form:
\begin{equation}\label{3}
{Y}_{jik}=\mathbf{f}(j)^\top\mbox{\boldmath{$\theta$}}_i + \varepsilon_{jik},\quad k=1, \dots, K,\quad i=N_{j-1}+1, \dots, N_j,\quad j=1, \dots, J.
\end{equation}
The parameter vectors $\mbox{\boldmath{$\theta$}}_i$ have the expected value $\mathrm{E}(\mbox{\boldmath{$\theta$}}_i)=\mbox{\boldmath{$\theta$}}_0:=(\mu, \alpha_{1}, \dots, \alpha_{J-1})^\top$ and the covariance matrix $\mathrm{Cov}(\mbox{\boldmath{$\theta$}}_i)=\sigma^2\textrm{block-diag}(u, v\,\mathbb{I}_{J-1})$, where $\mathbb{I}_p$ denotes the $p\times p$ identity matrix and $\textrm{block-diag}(\mathbf{A}_1, \dots, \mathbf{A}_p)$ is the block diagonal matrix with $l\times l$ blocks $\mathbf{A}_1, \dots, \mathbf{A}_p$.

On the individual level the vectors $\mathbf{Y}_{j,i}=({Y}_{ji1}, \dots, {Y}_{jiK})^\top$ for individuals in the $j$-th group can be specified as
\begin{equation}\label{5}
\mathbf{Y}_{j,i}=\mathds{1}_K\,\mathbf{f}(j)^\top\mbox{\boldmath{$\theta$}}_i + \mbox{\boldmath{$\varepsilon$}}_{j,i},\quad i=N_{j-1}+1, \dots, N_j,\quad j=1, \dots, J,
\end{equation}
where $\mathds{1}_K$ denotes the vector of length $K$ with all entries equal to $1$ and $\mbox{\boldmath{$\varepsilon$}}_{j,i}=(\varepsilon_{ji1}, \dots, \varepsilon_{jiK})^\top$.
% is the vector of the observational errors for the $i$-th individual in the $j$-th group.

On the group level for the group parameters $\mbox{\boldmath{$\vartheta$}}_{j}:= (\mbox{\boldmath{$\theta$}}_{N_{j-1}+1}^\top, \dots, \mbox{\boldmath{$\theta$}}_{N_j}^\top)^\top$ the vector $\mathbf{Y}_{j}=(\mathbf{Y}_{j,\,(j-1)n+1}^\top, \dots, \mathbf{Y}_{j,\,N_j}^\top)^\top$ of all observations at all individuals in the $j$-th group is given by
%\begin{equation}
%\mbox{\boldmath{$\theta$}}_{(j)}:= \left(\begin{array}{c} \mbox{\boldmath{$\theta$}}_{(j-1)n+1} \\ \vdots \\ \mbox{\boldmath{$\theta$}}_{N_j} \end{array}\right).
%\end{equation}
\begin{equation}\label{7}
\mathbf{Y}_{j}=\left(\mathbb{I}_{r_j}\otimes\left(\mathds{1}_K\,\mathbf{f}(j)^\top\right)\right)\mbox{\boldmath{$\vartheta$}}_{j} + \mbox{\boldmath{$\varepsilon$}}_{j},\quad j=1, \dots, J,
\end{equation}
where 
\begin{equation*}
r_j:=\left\{\begin{array}{ll} n\,, & j=1, \dots, J-1 \\ m\,, & j=J \end{array}\right.,
\end{equation*}
is the number of individuals in the $j$-th group, ``$\otimes$'' denotes the Kronecker product and $\mbox{\boldmath{$\varepsilon$}}_{j}=(\mbox{\boldmath{$\varepsilon$}}_{j,\,(j-1)n+1}^\top, \dots, \mbox{\boldmath{$\varepsilon$}}_{j,\,N_j}^\top)^\top$.
% is the vector of observational errors in group $j$.

Finally, we introduce the vector of all individual random parameters $\mbox{\boldmath{$\theta$}}= (\mbox{\boldmath{$\vartheta$}}_{1}^\top, \dots,\mbox{\boldmath{$\vartheta$}}_{J}^\top)^\top$ (or, equivalently, $\mbox{\boldmath{$\theta$}}:= (\mbox{\boldmath{$\theta$}}_1^\top, \dots, \mbox{\boldmath{$\theta$}}_N^\top)^\top$), for which 
%\begin{equation}\label{101}
%\mbox{\boldmath{$\theta$}}:= \left(\begin{array}{c} \mbox{\boldmath{$\theta$}}_1 \\ \vdots \\ \mbox{\boldmath{$\theta$}}_N \end{array}\right),
%\end{equation}
the full vector $\mathbf{Y}=(\mathbf{Y}_{1}^\top, \dots, \mathbf{Y}_{J}^\top)^\top$ of observations at all individuals in all groups is given by the following formula:
\begin{equation}\label{11}
\mathbf{Y}=\mathrm{Diag}_{j=1}^J\left(\mathbb{I}_{r_j}\otimes\left(\mathds{1}_K\,\mathbf{f}(j)^\top\right)\right)\mbox{\boldmath{$\theta$}} + \mbox{\boldmath{$\varepsilon$}},
\end{equation}
where $\mathrm{Diag}_{s=1}^p(\mathbf{A}_s)$ is the block diagonal matrix with $l_s\times t_s$ blocks $\mathbf{A}_s$, $s=1, \dots, p$,
%where\, $\mathrm{Diag}_{s=1, ..., p}(\mathbf{A}_s)$\, is the block matrix of the following structure:
%\begin{equation}
%\mathrm{Diag}_{s=1, ..., p}(\mathbf{A}_s)=\left(\begin{array}{cccc}\mathbf{A}_1 & \mathbf{0} & \dots & \mathbf{0} \\  \mathbf{0} & \mathbf{A}_2 & \dots & \mathbf{0} \\ \vdots & \vdots & \ddots & \vdots\\ \mathbf{0} & \mathbf{0} & \dots & \mathbf{A}_p
%\end{array}\right),
%\end{equation}
%where\, $\mathbf{A}_s$, $s=1, ..., p$, are some $l_s\times t_s$ matrices, $\mathbf{0}$ denote zero matrices, 
 and $\mbox{\boldmath{$\varepsilon$}}=(\mbox{\boldmath{$\varepsilon$}}_{1}^\top, \dots, \mbox{\boldmath{$\varepsilon$}}_{J}^\top)^\top$.

%Note that the expected value and the covariance matrix of the random effects vector\, $\mbox{\boldmath{$\theta$}}$\, are\, $\mathrm{E}(\mbox{\boldmath{$\theta$}})=\left(\mathds{1}_N\otimes \mathbb{I}_J\right) \mbox{\boldmath{$\theta$}}_0$\, and\, $\mathrm{Cov}(\mbox{\boldmath{$\theta$}})=\sigma^2\,\mathbb{I}_N\otimes\mathbf{D}$\,, respectively. The vector of observational errors\, $\mbox{\boldmath{$\varepsilon$}}$\, has zero expected value and the covariance matrix\, $\mathrm{Cov}(\mbox{\boldmath{$\varepsilon$}})=\sigma^2\,\mathbb{I}_{NK}$\,.

Alternatively, for the random vector $\mbox{\boldmath{$\zeta$}}:=\mbox{\boldmath{$\theta$}}-\left(\mathds{1}_N\otimes \mathbb{I}_J\right) \mbox{\boldmath{$\theta$}}_0$ the model (\ref{11}) can be represented in the form
\begin{equation}\label{12}
\mathbf{Y}=\mathrm{Vec}_{j=1}^J\left(\mathds{1}_{r_j}\otimes\left(\mathds{1}_K\,\mathbf{f}(j)^\top\right)\right)\mbox{\boldmath{$\theta$}}_0+\mathrm{Diag}_{j=1}^J\left(\mathbb{I}_{r_j}\otimes\left(\mathds{1}_K\,\mathbf{f}(j)^\top\right)\right)\mbox{\boldmath{$\zeta$}} + \mbox{\boldmath{$\varepsilon$}},
\end{equation}
where $\mathrm{Vec}_{s=1}^p\left(\mathbf{A}_s\right)=(\mathbf{A}_1^\top, \dots, \mathbf{A}_p^\top)^\top$ for some $l_s\times t$ matrices $\mathbf{A}_s$, $s=1, \dots, p$. 

Note that $\mbox{\boldmath{$\zeta$}}= (\mbox{\boldmath{$\zeta$}}_1^\top, \dots, \mbox{\boldmath{$\zeta$}}_N^\top)^\top$ for $\mbox{\boldmath{$\zeta$}}_i=\mbox{\boldmath{$\theta$}}_i-\mbox{\boldmath{$\theta$}}_0$, $i=1, \dots, N$. 
The expected value of $\mbox{\boldmath{$\zeta$}}$ is zero and $\mathrm{Cov}(\mbox{\boldmath{$\zeta$}})=\sigma^2\,\mathbb{I}_N\otimes\textrm{block-diag}(u,\, v\,\mathbb{I}_{J-1})$.

\section{Estimation and Prediction}\label{k3}

In this chapter we determine the best linear unbiased estimator (BLUE) for the population mean parameters $\alpha_j$ and $\mu$ and the best linear unbiased predictor (BLUP) for the individual random parameters $\alpha_{ji}$ and $\mu_{i}$ for $j=1, \dots, J-1$ and $i=1, \dots, N$. 

We denote the mean observation in group $j$ by $\bar{Y}_j=\frac{1}{r_j}\sum_{i=N_{j-1}+1}^{N_j}\bar{Y}_{j,i}$ and the mean observation at individual $i$ in group $j$ by $\bar{Y}_{j,i}=\frac{1}{K}\sum_{k=1}^{K}Y_{jik}$ for all $j=1, \dots, J$, $i=N_{j-1}+1, \dots, N_j$. Then we obtain the following results for the BLUEs of the mean parameters $\mu$ and $\alpha_j$ and the BLUPs of the treatment effects $\mu_i$ and $\alpha_{ji}$.
\begin{thm}\label{t1}
The BLUE for the population intercept parameter $\mu$ is given by 
\begin{equation}\label{mu}
\hat{\mu}=\bar{Y}_{J}
\end{equation}
and the BLUE for the population treatment effect $\alpha_j$ is given by
\begin{equation}
\hat{\alpha_j}= \bar{Y}_{j}-\bar{Y}_{J},\quad j=1, \dots, J-1.
\end{equation}
\end{thm}
\begin{thm}\label{t2}
If the $i$-th individual is in the control group, the BLUP for the individual intercepts $\mu_{i}$ is given by
\begin{equation}
\hat{\mu}_{i}=\frac{Ku}{Ku+1}\bar{Y}_{J,i}+\frac{1}{Ku+1}\bar{Y}_{J},\quad i=N_{J-1}+1, \dots, N,
\end{equation}
otherwise, if the $i$-th individual is in the $j$-th treatment group, the BLUP is given by
\begin{equation}
\hat{\mu}_{i}=\frac{Ku}{K(v+u)+1}\,(\bar{Y}_{j,i}-\bar{Y}_{j})+\bar{Y}_{J},\quad j=1, \dots, J-1, \quad i=N_{j-1}+1, \dots, N_j.
\end{equation}
If the $i$-th individual is in the $j$-th treatment group, the BLUP for the individual treatment effect $\alpha_{ji}$ is given by
\begin{equation}
\hat{\alpha}_{ji}=\frac{Kv}{K(v+u)+1}\,(\bar{Y}_{j,i}-\bar{Y}_{J})+\frac{Ku+1}{K(v+u)+1}\,(\bar{Y}_{j}-\bar{Y}_{J}),\quad j=1, \dots, J-1,\quad i=N_{j-1}+1, \dots, N_j
\end{equation}
otherwise the BLUP is given by
\begin{equation}\label{alpha}
\hat{\alpha}_{ji}=\bar{Y}_{j}-\bar{Y}_{J},\quad j=1, \dots, J-1,\quad i=1, \dots, N, \quad i\neq N_{j-1}+1, \dots, N_j.
%i\in\{1, \dots, N\}\setminus\{N_{j-1}+1, \dots, N_j\}.
\end{equation}
\end{thm}
%To determine the BLUPs for the individual parameters\, $\mu_{i}$\, and\, $\alpha_{ji}$\, for all\, $j=1, ..., J-1$\, and\, $i=1, ..., N$\,, we apply formulas (\ref{thetai}), (\ref{101}) and (\ref{gamma1}). The result is provided by the following theorem.
%Note again that each individual can appear in each group. Therefore, we predict the individual treatment effects\, $\alpha_{ji}$\, for all individuals\, $i=1, ..., N$\,.
For the vector $\mbox{\boldmath{$\Psi$}}_0:=(\alpha_{1}, \dots, \alpha_{J-1})^\top$ of the mean treatment effects the BLUE is given by $\hat{\mbox{\boldmath{$\Psi$}}}_0=(\hat{\alpha}_{1}, \dots, \hat{\alpha}_{J-1})^\top$. 
%The performance of this estimator may be measured using the covariance matrix of $\hat{\mbox{\boldmath{$\Psi$}}}_0$.
%\begin{thm}\label{t3}
%The covariance matrix of the BLUE\, $\hat{\mbox{\boldmath{$\Psi$}}}_0$\, is given by
%\begin{equation}\label{cov}
%\mathrm{Cov}\left(\hat{\mbox{\boldmath{$\Psi$}}}_{0}\right)=\sigma^2\left(\frac{Ku+1}{Km}\,\mathds{1}_{J-1}\,\mathds{1}_{J-1}^\top+\frac{K(v+u)+1}{Kn}\,\mathbb{I}_{J-1}\right).
%\end{equation}
%\end{thm}

We denote the vector of all individual treatment effects by $\mbox{\boldmath{$\Psi$}}:=(\mbox{\boldmath{$\Psi$}}_1^\top, \dots, \mbox{\boldmath{$\Psi$}}_N^\top)^\top$, where $\mbox{\boldmath{$\Psi$}}_i:=(\alpha_{1\,i}, \dots, \alpha_{J-1\,i})^\top$ is the vector of individual treatment effects for the $i$-th individual, $i=1, \dots, N$. Then the BLUP of $\mbox{\boldmath{$\Psi$}}$ is given by $\hat{\mbox{\boldmath{$\Psi$}}}=(\hat{\mbox{\boldmath{$\Psi$}}}_1^\top, \dots, \hat{\mbox{\boldmath{$\Psi$}}}_N^\top)^\top$, where $\hat{\mbox{\boldmath{$\Psi$}}}_i=(\hat{\alpha}_{1\,i}, \dots, \hat{\alpha}_{J-1\,i})^\top$ is the BLUP for $\mbox{\boldmath{$\Psi$}}_i$, $i=1, \dots, N$.  

%We consider the mean squared error (MSE) matrix of $\hat{\mbox{\boldmath{$\Psi$}}}$ as a performance measure of the predictor. 
The next theorems provide the covariance matrix of $\hat{\mbox{\boldmath{$\Psi$}}}_0$ and the MSE matrix of $\hat{\mbox{\boldmath{$\Psi$}}}$.
%\begin{equation}
%\mbox{\boldmath{$\Psi$}}_0:=\left(\begin{array}{c}\alpha_{1}\\ \vdots \\ \alpha_{J-1}\end{array}\right)
%\end{equation}
%\begin{equation}
%\mbox{\boldmath{$\Psi$}}_i:=\left(\begin{array}{c}\alpha_{1\,i}\\ \vdots \\ \alpha_{J-1\,i}\end{array}\right)
%\end{equation} 
%\begin{equation}
%\mbox{\boldmath{$\Psi$}}:=\left(\begin{array}{c}\mbox{\boldmath{$\Psi$}}_1^\top \\ \vdots \\ \mbox{\boldmath{$\Psi$}}_N^\top\end{array}\right).
%\end{equation}
%\begin{equation}\label{psih}
%\hat{\mbox{\boldmath{$\Psi$}}}=\left(\begin{array}{c}\hat{\mbox{\boldmath{$\Psi$}}}_1^\top \\ \vdots \\ \hat{\mbox{\boldmath{$\Psi$}}}_N^\top\end{array}\right)
%\end{equation}
\begin{thm}\label{t3}
The covariance matrix of the BLUE $\hat{\mbox{\boldmath{$\Psi$}}}_0$ is given by
\begin{equation}\label{cov}
\mathrm{Cov}\left(\hat{\mbox{\boldmath{$\Psi$}}}_{0}\right)=\sigma^2\left(\frac{Ku+1}{Km}\,\mathds{1}_{J-1}\,\mathds{1}_{J-1}^\top+\frac{K(v+u)+1}{Kn}\,\mathbb{I}_{J-1}\right).
\end{equation}
\begin{thm}\label{t4}
\end{thm}
The mean squared error matrix of the BLUP $\hat{\mbox{\boldmath{$\Psi$}}}$ is given by
\begin{equation}\label{mse}
\mathrm{Cov}\left(\hat{\mbox{\boldmath{$\Psi$}}}-\mbox{\boldmath{$\Psi$}}\right)= \mathbf{B}_{1}+\mathbf{B}_{2}+\mathbf{B}_{2}^\top+\mathbf{B}_{3},
\end{equation}
where
\begin{equation*}%\label{B1}
\mathbf{B}_{1}=\sigma^2\mathds{1}_N\mathds{1}_N^\top\otimes\left(\frac{Ku+1}{Km}\,\mathds{1}_{J-1}\,\mathds{1}_{J-1}^\top+\frac{K(v+u)+1}{Kn}\,\mathbb{I}_{J-1}\right),
\end{equation*}
\begin{equation*}
\mathbf{B}_{2}=-\sigma^2v\,\mathds{1}_N\otimes\left(\mathrm{tVec}_{j=1}^{J-1}\left(\frac{1}{n}\mathds{1}_n^\top\otimes\left(\mathbf{e}_j\mathbf{e}_j^\top\right)\right)\,\, \vdots \,\, \mathbf{0}_{(J-1)\times m(J-1)}\right),
\end{equation*}
where $\mathbf{0}_{p\times q}$ denotes the $p\times q$ zero matrix and $\mathrm{tVec}_{s=1}^{p}\left(\mathbf{A}_s\right)=\left(\mathrm{Vec}_{s=1}^{p}\left(\mathbf{A}_s^\top\right)\right)^\top$, and
\begin{equation*}%\label{B3}
\mathbf{B}_{3}=\sigma^2v\,\mathrm{block\textrm{-}diag}\left(\mathbb{I}_{n(J-1)^2}-\frac{Kv}{K(v+u)+1}\mathrm{Diag}_{j=1}^{J-1}\left(\left(\mathbb{I}_n-\frac{1}{n}\mathds{1}_n\mathds{1}_n^\top\right)\otimes\left(\mathbf{e}_j\mathbf{e}_j^\top\right) \right),\, \mathbb{I}_{m(J-1)} \right).
\end{equation*}
%\begin{equation}
%\mathbf{B}_{3}=-\sigma^2v\left(\begin{array}{cc}\mathbb{I}_{n(J-1)^2}-\frac{Kv}{K(v+u)+1}\mathrm{Diag}_{j=1, ..., J-1}\left(\left(\mathbb{I}_n-\frac{1}{n}\mathds{1}_n\mathds{1}_n^\top\right)\otimes\left(\mathbf{e}_j\mathbf{e}_j^\top\right) \right) & \mathbf{0}_{\ell_1\times \ell_2} \\ \mathbf{0}_{\ell_2\times \ell_1} & \mathbb{I}_{m(J-1)} \end{array}\right)
%\end{equation}
%for\, $\ell_1=n(J-1)^2$\, and\, $\ell_2=m(J-1)$\,.
\end{thm}
For the proofs of Theorems~\ref{t1}-\ref{t4} see Appendix~\ref{A1}.

%\begin{equation}\label{9}
%\mathbf{Y}=\left(\begin{array}{cc}\mathrm{Diag}_{j=1, ..., J-1}\left(\mathbb{I}_n\otimes\left(\mathds{1}_K\,(\mathbf{e}_j+\mathbf{e}_J)^\top\right)\right) & 0 \\ 0 & \mathbb{I}_m\otimes\left(\mathds{1}_K\,\mathbf{e}_J^\top\right)\end{array}\right)\mbox{\boldmath{$\theta$}} + \mbox{\boldmath{$\varepsilon$}},
%\end{equation}

\section{Optimal Design}

In this chapter we optimize the numbers $n$ and $m$ of individuals in the treatment and control groups, respectively. We define the exact experimental design as

\begin{equation}
\xi = \left(\begin{array}{cccc}1 & \dots & J-1 & J \\ n & \dots & n & m \end{array}\right),
\end{equation}
where the indexes $1$, \dots, $J-1$ denote the treatment groups and the index $J$ is used for the control group. 

For analytical purposes, we also define the approximate design:
\begin{equation}\label{adesign}
\xi = \left(\begin{array}{cccc}1 & \dots & J-1 & J \\ w & \dots & w & 1-(J-1)w \end{array}\right),\quad w\in \left(0, \frac{1}{J-1}\right),
\end{equation}
where $w=\frac{n}{N}$ is the weight of a treatment group and $1-(J-1)w=\frac{m}{N}$ is the weight of the control group. Then only the optimal weight $w^*$ of a treatment group has to be determined. 
%Note that the resulting value\, $n^*=w^*N$\, is not necessarily integer and should be rounded in some appropriate way.

\subsection{\textit{A}-criterion}\label{4.1}

For the estimation of the population treatment effects $\mbox{\boldmath{$\Psi$}}_0$ the \textit{A}-criterion for an exact design is defined as the trace of the covariance matrix of the BLUE $\hat{\mbox{\boldmath{$\Psi$}}}_0$:
\begin{equation*}
\mathrm{A}_{\Psi_0}(\xi) = \textrm{tr}\left( \textrm{Cov}\left(\hat{\mbox{\boldmath{$\Psi$}}}_0\right) \right).
\end{equation*}
We determine the trace of the covariance matrix \eqref{cov},
%\, $\textrm{Cov}\left(\hat{\mbox{\boldmath{$\Psi$}}}_0\right)$\,, 
replace $n$ by $Nw$ and $m$ by $N(1-w)$ and obtain the next results for an approximate design.
\begin{thm}
The \textit{A}-criterion for the estimation of the population treatment effects $\mbox{\boldmath{$\Psi$}}_0$ is given for an approximate design by
\begin{equation}\label{acrp}
\mathrm{A}_{\Psi_0}(w) = \frac{\sigma^2(J-1)}{K\,N}\left(\frac{K(v+u)+1}{w}+\frac{Ku+1}{1-(J-1)w}\right).
\end{equation}
\end{thm}
\begin{thm}
The \textit{A}-optimal weight $w^*_{A,\Psi_0}$ for the estimation of the population treatment effects is given by
\begin{equation}\label{wap}
w^*_{A,\Psi_0} = \frac{1}{J-1+\sqrt{J-1}\sqrt{\frac{Ku+1}{K(v+u)+1}}}.
\end{equation}
\end{thm}
Note that for large values of the intercepts variance ($u\rightarrow\infty$) the optimal weight \eqref{wap} tends to the value $w^*_{A,\Psi_0} = \frac{1}{J-1+\sqrt{J-1}}$, which coincides with the optimal weight in the fixed effects model (see e.g. \cite{bai}, ch.~3 or \cite{schw}, ch.~3) . If the treatment effects variance takes a very large value ($v\rightarrow\infty$), the limiting optimal design assigns all observations to be taken in the treatment groups: $w^*_{A,\Psi_0} = \frac{1}{J-1}$. If both variances are large and the variance ratio $b=v/u$ is fixed, the limiting optimal design is given by 
\begin{equation*}
w^*_{A,\Psi_0} = \frac{1}{J-1+\sqrt{J-1}\sqrt{\frac{1}{1+b}}}.
\end{equation*}

The \textit{A}-criterion for the prediction of the individual treatment effects $\mbox{\boldmath{$\Psi$}}_i$ is defined for an exact design as the trace of the mean squared error matrix of the BLUP $\hat{\mbox{\boldmath{$\Psi$}}}$:
\begin{equation*}
\mathrm{A}_{\Psi}(\xi) = \textrm{tr}\left( \textrm{Cov}\left(\hat{\mbox{\boldmath{$\Psi$}}}-\mbox{\boldmath{$\Psi$}}\right) \right).
\end{equation*}
The next theorem presents the \textit{A}-criterion for an approximate design.
\begin{thm}\label{t7}
The \textit{A}-criterion for the prediction of the individual treatment effects $\mbox{\boldmath{$\Psi$}}_i$ is given for an approximate design by
\begin{equation}\label{acri}
\mathrm{A}_{\Psi}(w) = \sigma^2(J-1)\left(\frac{K(v+u)+1}{Kw}+\frac{Ku+1}{K(1-(J-1)w)}+v\left(N-2-\frac{Kv\,(Nw-1)}{K(v+u)+1}\right)\right).
\end{equation}
\end{thm}
\begin{proof}
The result of Theorem~\ref{t7} follows from (\ref{mse}) and
\begin{equation*}
\mathrm{tr}(\mathbf{B}_1) = \frac{\sigma^2(J-1)}{K}\left(\frac{K(v+u)+1}{w}+\frac{Ku+1}{(1-(J-1)w)}\right),
\end{equation*}
\begin{equation*}
\mathrm{tr}(\mathbf{B}_2) = -\sigma^2(J-1)\,v
\end{equation*}
and
\begin{equation*}
\mathrm{tr}(\mathbf{B}_3) = \sigma^2(J-1)\,v\left(N-\frac{Kv\,(Nw-1)}{K(v+u)+1}\right).
\end{equation*}
\end{proof}
Note that there is no explicit formula for the optimal weight in this case. However, it is easy to see that there is a unique solution $w^*$, which may be determined numerically for given values of $N$, $J$, $K$, $u$ and $v$. To illustrate the behavior of optimal designs, we consider a numerical example.

\textbf{Example 1.} 
Let the total number of individuals be $N=100$ and the number of observations per individual be $K=10$ and the variance ratio $b=v/u$ be fixed by the values $2$, $0.6$ and $0.001$. The next graphics (Figure 1 and Figure 2) illustrate the behavior of the \textit{A}-optimal weight in dependence on the treatment effects variance for the special cases of one (left panel) and two (right panel) treatment groups ($J=2$ and $J=3$, respectively). The parameter $\rho=v/(1+v)$ is used instead of the variance parameter $v$ to cover all values of the treatment effects variance by a finite interval. On the graphics the solid, dashed and dotted lines present the optimal weight for the values $2$, $0.6$ and $0.001$ of the ratio $b$. Note that the optimal weight $w^*$ takes all its values in the intervals $(0,1)$ and $(0,0.5)$ in the models with one and two treatment groups, respectively. 
\begin{figure}[ht]
    \begin{minipage}[]{8.2 cm}
       \centering
       \includegraphics[width=78mm]{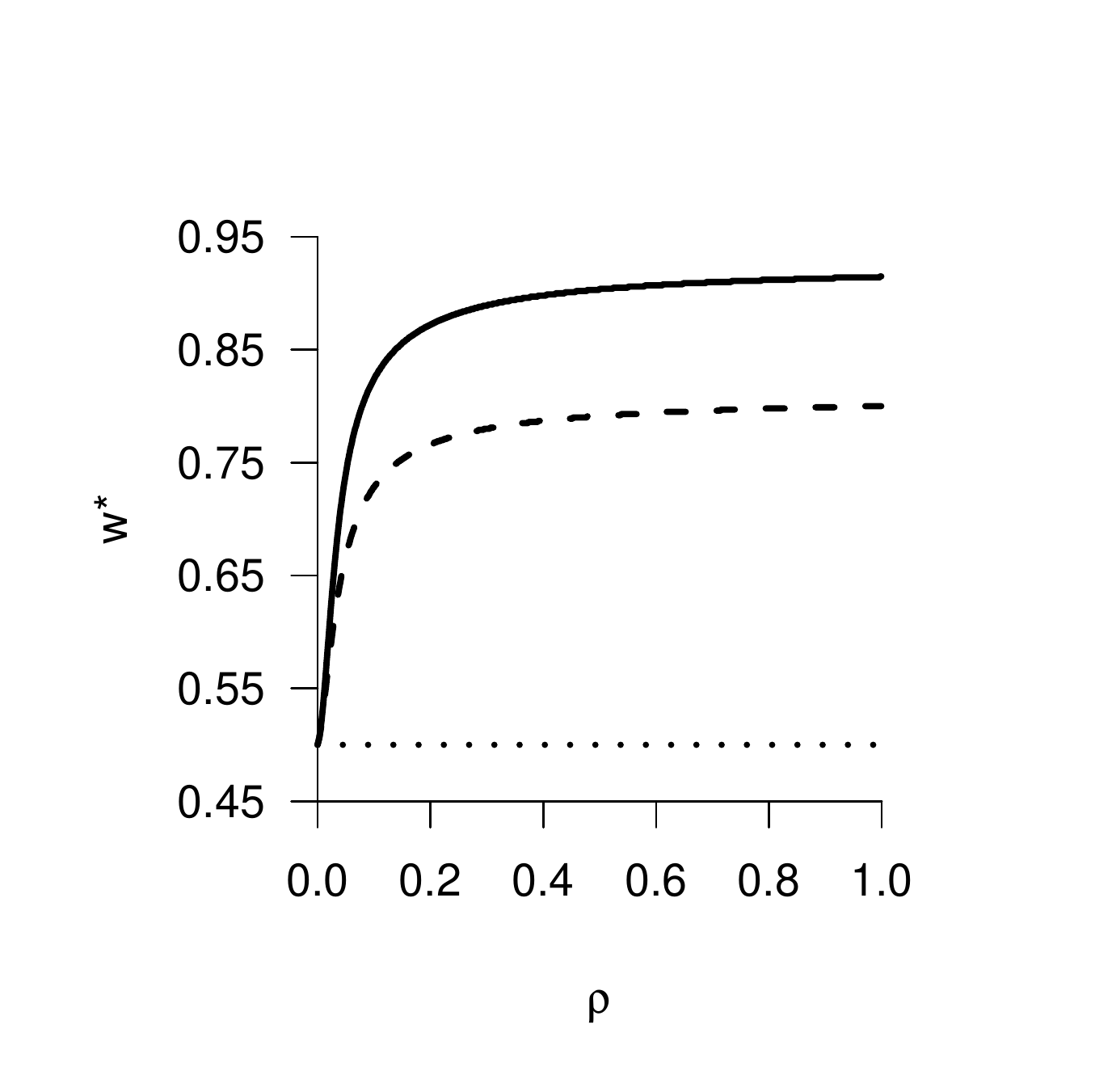}
       \label{a1}
       \end{minipage}
       \begin{minipage}[]{8.2 cm}
       \centering
       \includegraphics[width=78mm]{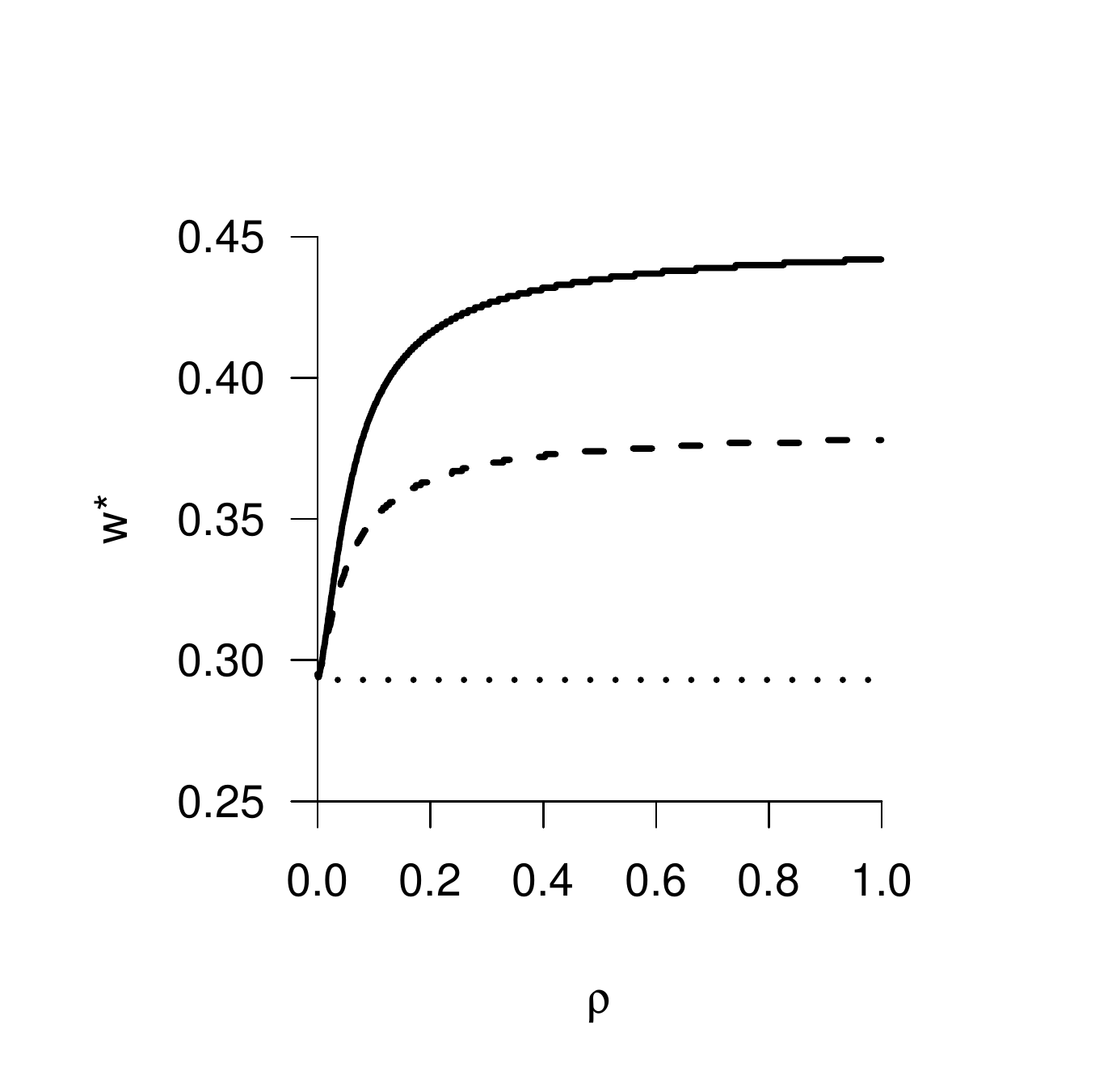}
       \label{a2}
       \end{minipage}
       \vspace{-1mm}
       \hspace*{5 mm}
       \begin{minipage}[]{6.5 cm}
       Figure 1: \textit{A}-optimal weight $w^*$ for one treatment group and variance ratios $b=2$ (solid line), $b=0.6$ (dashed line) and $b=0.001$ (dotted line)
       \end{minipage}
       \hspace{15 mm}
       \begin{minipage}[]{6.5 cm}
       Figure 2: \textit{A}-optimal weight $w^*$ for two treatment groups and variance ratios $b=2$ (solid line), $b=0.6$ (dashed line) and $b=0.001$ (dotted line)
       \end{minipage}
    \end{figure} 
			
For the models with one and two treatment groups the optimal the optimal weights start (for $\rho \rightarrow 0$) at points $w^*=0.5$ and $w^*\approx0.29$, respectively. This may be explained by the fact that the optimal designs for the fixed effects models are equal to $w^*_{A,fix} = \frac{1}{J-1+\sqrt{J-1}}$ and result in $w^*_{A,fix} = 0.5$ for $J=2$ and $w^*_{A,fix} = \frac{1}{2+\sqrt{2}}\approx0.29$ for $J=3$. The optimal weights increase with increasing variance of the individual treatment effects with limiting values (for $\rho\rightarrow 1$) $w^*=0.91$, $w^*=0.80$ and $w^*\approx0.50$ for $J=2$ and $w^*=0.44$, $w^*=0.38$ and $w^*\approx0.29$ for $J=3$ for $b=2$, $b=0.6$ and $b=0.001$, respectively.

Figure 3 and Figure 4 present the efficiency of the optimal weight $w^*_{A,fix}$ from the fixed effects model for the present model for one (left panel) and two (right panel) treatment groups for the values $2$, $0.6$ and $0.001$ of the ratio $b$.

\begin{figure}[ht]
    \begin{minipage}[]{8.2 cm}
       \centering
       \includegraphics[width=78mm]{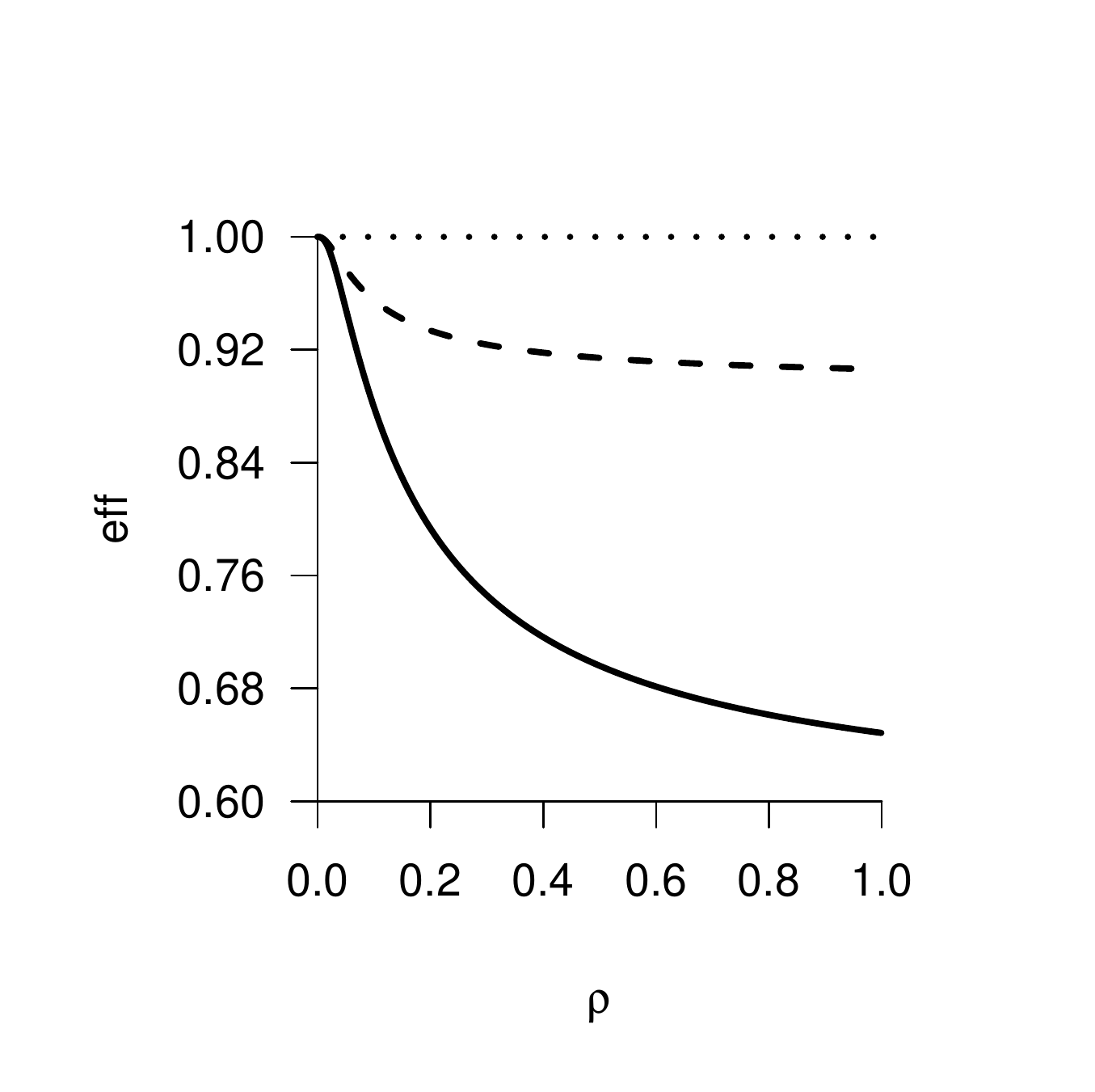}
       \label{a1eff}
       \end{minipage}
       \begin{minipage}[]{8.2 cm}
       \centering
       \includegraphics[width=78mm]{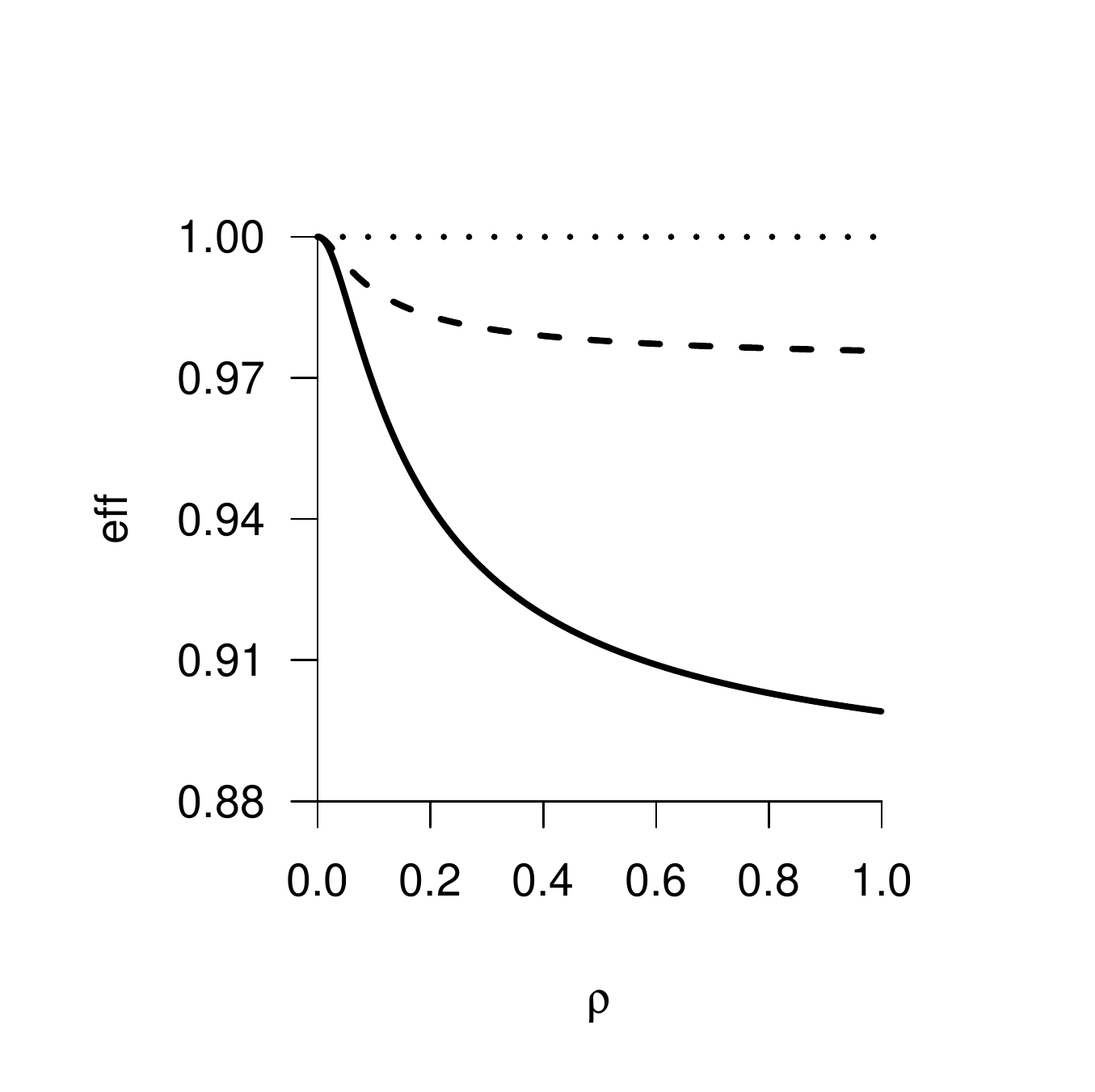}
       \label{a2eff}
       \end{minipage}
       \vspace{-1mm}
       \hspace*{5 mm}
       \begin{minipage}[]{6.5 cm}
       Figure 3: Efficiency of \textit{A}-optimal design in fixed effects model for one treatment group and variance ratios $b=2$ (solid line), $b=0.6$ (dashed line) and $b=0.001$ (dotted line)
       \end{minipage}
       \hspace{15 mm}
       \begin{minipage}[]{6.5 cm}
       Figure 4: Efficiency of \textit{A}-optimal design in fixed effects model for two treatment groups and variance ratios $b=2$ (solid line), $b=0.6$ (dashed line) and $b=0.001$ (dotted line)
       \end{minipage}
    \end{figure} 
		
For both particular models the efficiencies start at point $1$ and decrease with limits $\textit{eff}=0.65$, $\textit{eff}=0.90$ and $\textit{eff}\approx1$ for $J=2$ and $\textit{eff}=0.89$, $\textit{eff}=0.98$ and $\textit{eff}\approx1$ for $J=3$ for $b=2$, $b=0.6$ and $b=0.001$.

\subsection{D- and E-criterion}
In this section we consider \textit{D}- and \textit{E}-optimality criteria for the estimation and the prediction in multiple group models. We consider the general case of model (\ref{12}) for the estimation of population parameters and we restrict ourselves to the special case $J=2$ for the prediction of individual treatment effects.

For further considerations we will use the following result.

\begin{lem}\label{l1}
The eigenvalues of the covariance matrix of the BLUE $\hat{\mbox{\boldmath{$\Psi$}}}_0$ are
\begin{equation*}
\lambda_1=\frac{\sigma^2}{K}\left(\frac{(Ku+1)(J-1)}{\left(N-(J-1)\,n\right)}+\frac{K(v+u)+1}{n}\right)
\end{equation*}
with algebraic multiplicity $1$ and
\begin{equation*}
\lambda_2=\frac{\sigma^2(K(v+u)+1)}{K\,n}
\end{equation*}
with algebraic multiplicity $J-2$.
\end{lem}

\begin{proof}
To determine the eigenvalues of the covariance matrix $\textrm{Cov}\left(\hat{\mbox{\boldmath{$\Psi$}}}_0\right)$ we solve the equation
\begin{equation}\label{eig}
\textrm{det}\left( \textrm{Cov}\left(\hat{\mbox{\boldmath{$\Psi$}}}_0\right)-\lambda\,\mathbb{I}_{J-1} \right)=0,
\end{equation}
where $\lambda$ denotes an eigenvalue of $\textrm{Cov}\left(\hat{\mbox{\boldmath{$\Psi$}}}_0\right)$.
\begin{eqnarray*}
&&\textrm{det}\left( \textrm{Cov}\left(\hat{\mbox{\boldmath{$\Psi$}}}_0\right)-\lambda\,\mathbb{I}_{J-1} \right)\\
&&\quad=\,\,\, \textrm{det} \left(\frac{\sigma^2(Ku+1)}{K\left(N-(J-1)\,n\right)}\mathds{1}_{J-1}\,\mathds{1}_{J-1}^\top+\left(\frac{\sigma^2(K(v+u)+1)}{K\,n}-\lambda\right)\,\mathbb{I}_{J-1}\right)\\
&&\quad=\,\,\, \left(\frac{\sigma^2}{K}\left(\frac{(Ku+1)(J-1)}{N-(J-1)\,n}+\frac{K(v+u)+1}{n}\right)-\lambda\right)\left(\frac{\sigma^2(K(v+u)+1)}{K\,n}-\lambda\right)^{J-2}.
\end{eqnarray*}
Then
\begin{equation*}
\lambda_1=\frac{\sigma^2}{K}\left(\frac{(Ku+1)(J-1)}{\left(N-(J-1)\,n\right)}+\frac{K(v+u)+1}{n}\right)
\end{equation*}
and
\begin{equation*}
\lambda_2=\frac{\sigma^2(K(v+u)+1)}{K\,n}
\end{equation*}
are the solutions of (\ref{eig}).
\end{proof}

For the estimation of the population treatment effects the \textit{D}-criterion is defined as the logarithm of the determinant of the covariance matrix of the BLUE $\hat{\mbox{\boldmath{$\Psi$}}}_0$:
\begin{equation*}
\mathrm{D}_{\Psi_0}(\xi) = \textrm{ln}\,\textrm{det}\left( \textrm{Cov}\left(\hat{\mbox{\boldmath{$\Psi$}}}_0\right) \right).
\end{equation*}
for an exact design. We compute the determinant as the product of the eigenvalues, which are given in Lemma~\ref{l1}, and receive using $n=Nw$ the following result for approximate designs.

\begin{thm}
The \textit{D}-criterion for the estimation of the population treatment effects $\mbox{\boldmath{$\Psi$}}_0$ is given for an approximate design by
\begin{equation}\label{dcrp}
\mathrm{D}_{\Psi_0}(w) = c+\mathrm{ln}\left(\frac{K(v+u)+1}{w}+\frac{(J-1)(Ku+1)}{1-(J-1)w}\right)+(J-2)\,\mathrm{ln}\left(\frac{K(v+u)+1}{w}\right),
\end{equation}
where $c=(J-1)\,\mathrm{ln}\,\left(\frac{\sigma^2}{K\,N}\right)$.
\end{thm}
%\begin{proof}
%\begin{eqnarray*}
%\mathrm{D}_{\Psi_0}(\xi) &=& \textrm{ln}\,\textrm{det}\left( \textrm{Cov}\left(\hat{\mbox{\boldmath{$\Psi$}}}_0\right) \right)\\
%&=& \textrm{ln}\left(\left(\frac{\sigma^2}{K\,N}\right)^{J-1}\right)\,+\textrm{ln}\,\textrm{det} \left(\frac{(Ku+1)}{1-(J-1)\frac{n}{N}}\mathds{1}_{J-1}\,\mathds{1}_{J-1}^\top+\frac{K(v+u)+1}{\frac{n}{N}}\,\mathbb{I}_{J-1}\right)\\
%&=& c_1+\textrm{ln}\left(\left(\frac{(Ku+1)(J-1)}{1-(J-1)\frac{n}{N}}+\frac{K(v+u)+1}{\frac{n}{N}}\right)\left(\frac{K(v+u)+1}{\frac{n}{N}}\right)^{J-2}\right)\\
%&=& c_1+\textrm{ln}\left(\frac{(Ku+1)(J-1)}{1-(J-1)\frac{n}{N}}+\frac{K(v+u)+1}{\frac{n}{N}}\right)+(J-2)\,\textrm{ln}\left(\frac{K(v+u)+1}{\frac{n}{N}}\right).
%\end{eqnarray*}
%After applying $n=Nw$ we obtain the result (\ref{dcrp}).
%\end{proof}
\begin{thm}
The \textit{D}-optimal weight $w^*_{D,\Psi_0}$ for the estimation of the population treatment effects is given by
\begin{equation}\label{wdp}
w^*_{D,\Psi_0} = \frac{J-2+z}{(J-1)\left(J+z\right)},
\end{equation}
where $z=\sqrt{\frac{4(J-1)Kv}{Ku+1}+J^2}$.
%\begin{equation}\label{wdp}
%w^*_{D,\Psi_0} = \frac{J(Ku+1)+2(J-1)Kv-\sqrt{Ku+1}\sqrt{J^2(Ku+1)+4(J-1)Kv}}{2(J-1)^2Kv}.
%\end{equation}
\end{thm}

The \textit{E}-criterion for the estimation of the population treatment effects is defined for an exact design as the largest eigenvalue of the covariance matrix of the BLUE $\hat{\mbox{\boldmath{$\Psi$}}}_0$
\begin{equation*}
\mathrm{E}_{\Psi_0}(\xi) = \lambda_{max}\left( \textrm{Cov}\left(\hat{\mbox{\boldmath{$\Psi$}}}_0\right) \right),
\end{equation*}
where $\lambda_{max}(A)$ denotes the largest eigenvalue of the matrix $A$.

Using Lemma~\ref{l1} we receive the following form of the \textit{E}-criterion for approximate designs.
\begin{thm}
The \textit{E}-criterion for the estimation of the population treatment effects $\mbox{\boldmath{$\Psi$}}_0$ is given for an approximate design by
\begin{equation}\label{ecrp}
\mathrm{E}_{\Psi_0}(w) =\frac{\sigma^2}{K\,N}\left(\frac{K(v+u)+1}{w}+\frac{(J-1)(Ku+1)}{1-(J-1)w}\right).
\end{equation}
\end{thm}

\begin{thm}
The \textit{E}-optimal weight $w^*_{E,\Psi_0}$ for the estimation of the population treatment effects is given by
\begin{equation}\label{wep}
w^*_{E,\Psi_0} = \frac{1}{(J-1)\left(1+\sqrt{\frac{Ku+1}{K(v+u)+1}}\right)}.
\end{equation}
\end{thm}

Note that also for the \textit{D}- and \textit{E}-criteria the optimal weights for the estimation of the population parameters (\eqref{wdp} and \eqref{wep}) tend to the optimal weights in the fixed effects model: $w^*_{D,\Psi_0} \rightarrow \frac{1}{J}$ and $w^*_{E,\Psi_0} \rightarrow \frac{1}{2(J-1)}$, for $u\rightarrow\infty$. For large values of the treatment effects variance ($v\rightarrow\infty$) all observations should to be taken in the treatment groups: $w^*_{D,\Psi_0} =w^*_{E,\Psi_0} = \frac{1}{J-1}$. If both variances are large and $b=v/u$, the limiting values for the optimal weights are
\begin{equation*}
w^*_{D,\Psi_0} \rightarrow \frac{J-2+\sqrt{4(J-1)b+J^2}}{(J-1)\left(J+\sqrt{4(J-1)b+J^2}\right)}
\end{equation*} 
and 
\begin{equation*}
w^*_{E,\Psi_0} \rightarrow \frac{1}{(J-1)\left(1+\sqrt{\frac{1}{1+b}}\right)}.
\end{equation*}

For the prediction of the individual treatment effects we consider the particular multiple group model with one treatment group and one control group ($J=2$). In this case the mean squared error matrix (\ref{mse}) of the prediction simplifies to
\begin{equation}\label{mse2}
\mathrm{Cov}\left(\hat{\mbox{\boldmath{$\Psi$}}}-\mbox{\boldmath{$\Psi$}}\right)= \left(\begin{array}{cc} \mathbf{H}_{11} & \mathbf{H}_{12} \\ 
\mathbf{H}_{12}^\top & \mathbf{H}_{22} \end{array}\right),
\end{equation}
where
\begin{equation*}
\mathbf{H}_{11}=\sigma^2(Ku+1)\left(\frac{N}{Kn\,m}\mathds{1}_{n}\mathds{1}_{n}^\top+\frac{v}{K(u+v)+1}\left(\mathbb{I}_{n}-\frac{1}{n}\mathds{1}_{n}\mathds{1}_{n}^\top\right)\right),
\end{equation*}
\begin{equation*}
\mathbf{H}_{12}=\sigma^2(Ku+1)\frac{N}{Kn\,m}\mathds{1}_{n}\mathds{1}_{m}^\top,
\end{equation*}
\begin{equation*}
\mathbf{H}_{22}=\sigma^2\left(\left(\frac{K(u+v)+1}{Kn}+\frac{Ku+1}{Km}\right)\mathds{1}_{m}\mathds{1}_{m}^\top+v\,\mathbb{I}_{m}\right).
\end{equation*}
The approximate design (\ref{adesign}) simplifies to
\begin{equation*}
\xi = \left(\begin{array}{cc}1 & 2 \\ w & 1-w \end{array}\right).
\end{equation*}
The next Lemma provides the eigenvalues of the mean squared error matrix \eqref{mse2}.
\begin{lem}\label{l2}
The eigenvalues of the mean squared error matrix of the BLUP $\hat{\mbox{\boldmath{$\Psi$}}}$ are
\begin{equation*}
\lambda_1=\frac{\sigma^2\,v\,(Ku+1)}{K(v+u)+1},
\end{equation*}
\begin{equation*}
\lambda_2=\sigma^2v,
\end{equation*}
\begin{equation*}
\lambda_3=\frac{\sigma^2}{2}\left(\frac{v}{w}+\frac{Ku+1}{Kw(1-w)}+\frac{\sqrt{s_w}}{Kw(1-w)}\right),
\end{equation*}
where $s_w=K^2(1-w)^2v^2+2K(1-w)(1-2w)(Ku+1)v+(Ku+1)^2$, and
\begin{equation*}
\lambda_4=\frac{\sigma^2}{2}\left(\frac{v}{w}+\frac{Ku+1}{Kw(1-w)}-\frac{\sqrt{s_w}}{Kw(1-w)}\right)
\end{equation*}
with algebraic multiplicities $n-1$, $m-1$, $1$ and $1$, respectively.
\end{lem}
For the proof see Appendix~\ref{A2}.

We define the D-criterion for the prediction as the logarithm of the determinant of the mean squared error matrix of $\hat{\mbox{\boldmath{$\Psi$}}}$:
\begin{equation*}
\mathrm{D}_{\Psi}(\xi) = \mathrm{ln}\,\mathrm{det}\left( \textrm{Cov}\left(\hat{\mbox{\boldmath{$\Psi$}}}-\mbox{\boldmath{$\Psi$}}\right) \right).
\end{equation*}
We compute the determinant using the results of Lemma~\ref{l2} and obtain the following criterion for approximate designs.
\begin{thm}\label{t8}
The \textit{D}-criterion for the prediction of the individual treatment effects $\mbox{\boldmath{$\Psi$}}_i$ is given for an approximate design by
\begin{equation}\label{dcri}
\mathrm{D}_{\Psi}(w) = d+N\,w\,\mathrm{ln}\left(\frac{Ku+1}{K(u+v)+1}\right)-\mathrm{ln}\left(w\,(1-w)\right),
\end{equation}
where $d=\mathrm{ln}\left((\sigma^2)^Nv^{N-1}\,(K(u+v)+1)\frac{1}{K}\right)$.
\end{thm}
\begin{thm}
The \textit{D}-optimal weight $w^*_{D,\Psi}$ for the prediction of the individual treatment effects is given by
\begin{equation}\label{wdi}
w^*_{D,\Psi} = \frac{1}{N\,t}+\frac{1}{2}+\sqrt{\frac{1}{N^2\,t^2}+\frac{1}{4}} ,
\end{equation}
where $t=\mathrm{ln}\left(\frac{Ku+1}{K(u+v)+1}\right)$.
\end{thm}
We define the \textit{E}-criterion for the prediction as the largest eigenvalue of the mean squared error matrix:
\begin{equation*}
\mathrm{E}_{\Psi}(\xi) = \lambda_{max}\left( \textrm{Cov}\left(\hat{\mbox{\boldmath{$\Psi$}}}-\mbox{\boldmath{$\Psi$}}\right) \right).
\end{equation*}
The E-criterion for approximate designs follows directly from Lemma~\ref{l2}.
\begin{thm}\label{t9}
The \textit{E}-criterion for the prediction of the individual treatment effects $\mbox{\boldmath{$\Psi$}}_i$ is given for an approximate design by
\begin{equation}\label{ecri}
\mathrm{E}_{\Psi}(w) = \frac{\sigma^2}{2}\left(\frac{v}{w}+\frac{Ku+1}{Kw(1-w)}+\frac{\sqrt{s_w}}{Kw(1-w)}\right).
\end{equation}
%where $s_w=K^2(1-w)^2v^2+2K(1-w)(1-2w)(Ku+1)v+(Ku+1)^2$.
\end{thm}

\begin{thm}
The \textit{E}-optimal weight $w^*_{E,\Psi}$ for the prediction of the individual treatment effects is given by
\begin{equation}\label{wei}
w^*_{E,\Psi} =\frac{K(2u+v)+2}{K(4u+v)+4}.
\end{equation}
\end{thm}

Note that for both \textit{D}- and \textit{E}-criteria for the prediction the optimal weights tend to $w^*_{D,\Psi}=w^*_{E,\Psi}=0.5$, which is optimal for the fixed effects model (see e.g. \cite{wie}, p.~44), for $u\rightarrow\infty$ and to $w^*_{D,\Psi}=w^*_{E,\Psi}=1$ (all observations are being taken in the treatment group) for $v\rightarrow\infty$. For large $u$, large $v$ and fixed ratio $b=v/u$ we observe 
\begin{equation*}
w^*_{D,\Psi} \rightarrow -\frac{1}{N\,\mathrm{ln}(1+b)}+\frac{1}{2}+\sqrt{\frac{1}{N^2\left(\mathrm{ln}(1+b)\right)^2}+\frac{1}{4}}
\end{equation*}
and
\begin{equation*}
w^*_{E,\Psi} \rightarrow \frac{2/b+1}{4/b+1}.
\end{equation*}

\textbf{Example 2.} We consider again the model with $N=100$ individuals and $K=10$ observations per individual. The variance ratio $b$ takes the values $2$, $0.6$ and $0.001$ for both \textit{D}-criterion and \textit{E}-criteria. The number of treatment groups is fixed by one ($J=2$). 
The next picture (Figure 5 and Figure 6) illustrates the behavior of the \textit{D}- (left panel) and \textit{E}- (right panel) optimal weights in dependence on the variance parameter $\rho$. 

\begin{figure}[ht]
    \begin{minipage}[]{8.2 cm}
       \centering
       \includegraphics[width=78mm]{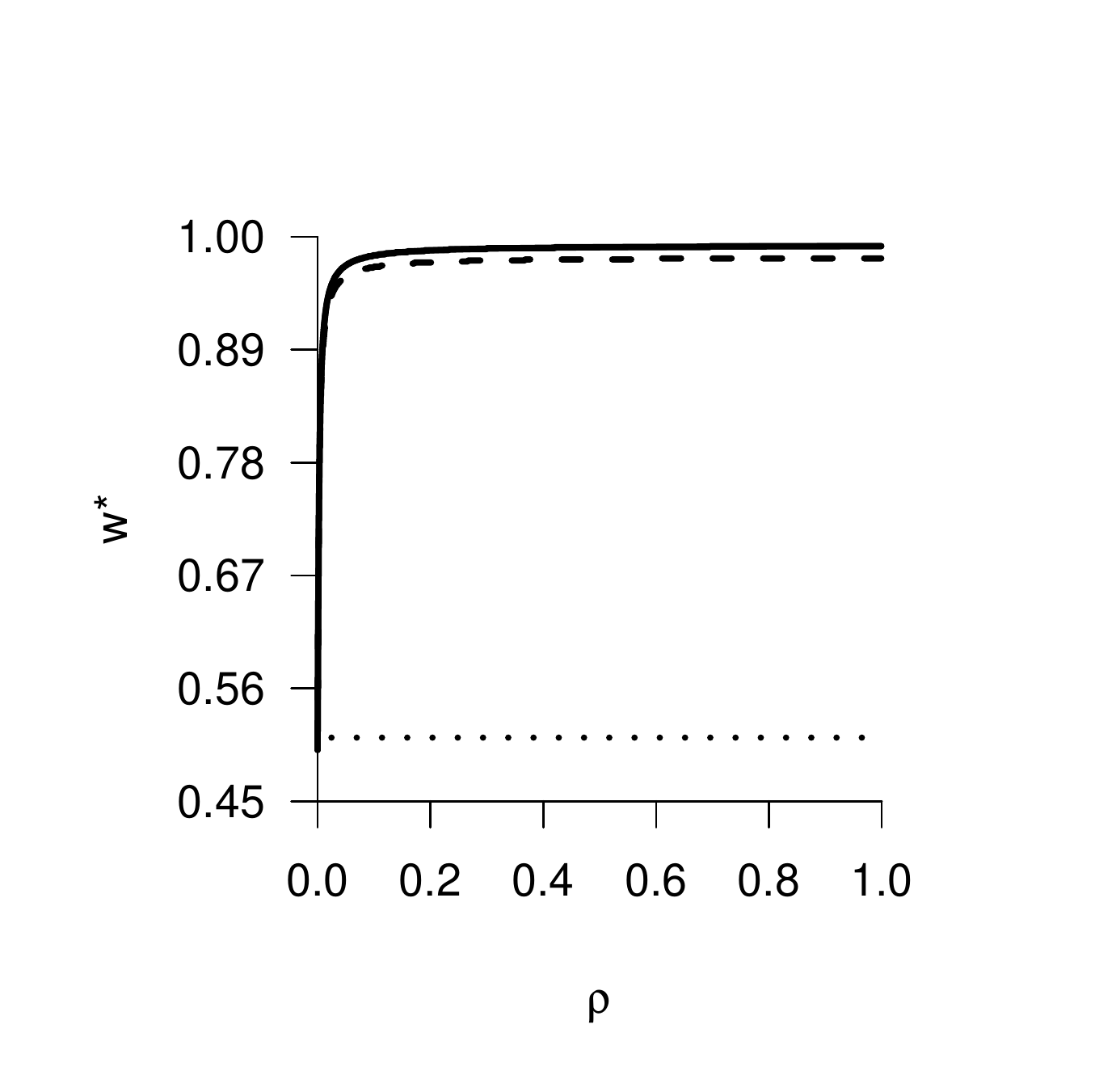}
       \label{d1}
       \end{minipage}
       \begin{minipage}[]{8.2 cm}
       \centering
       \includegraphics[width=78mm]{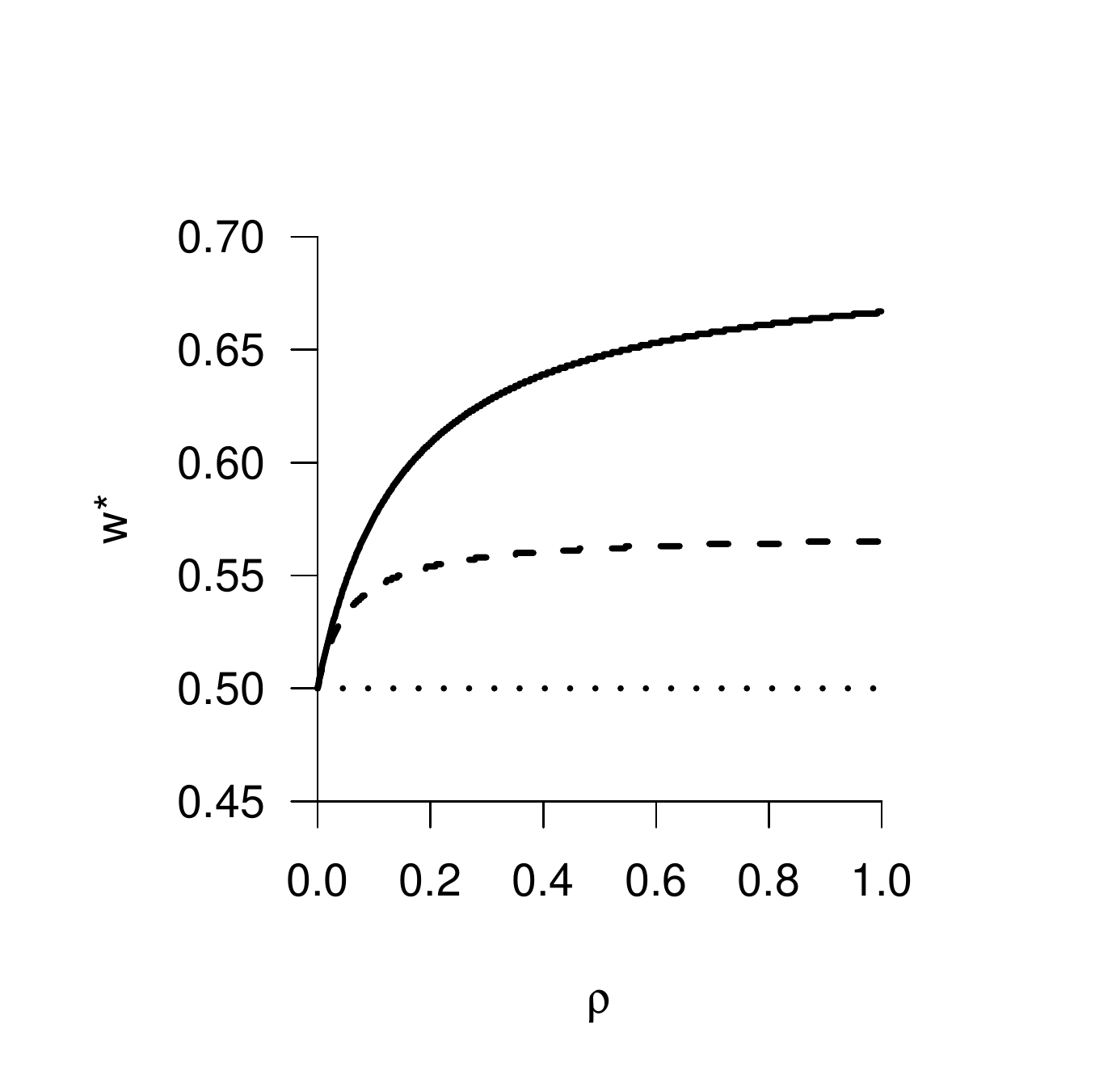}
       \label{e1}
       \end{minipage}
       \vspace{-1mm}
       \hspace*{5 mm}
       \begin{minipage}[]{6.5 cm}
       Figure 5: \textit{D}-optimal weight $w^*$ for variance ratios  $b=2$ (solid line), $b=0.6$ (dashed line) and $b=0.001$ (dotted line)
       \end{minipage}
       \hspace{15 mm}
       \begin{minipage}[]{6.5 cm}
       Figure 6: \textit{E}-optimal weight $w^*$ for variance ratios $b=2$ (solid line), $b=0.6$ (dashed line) and $b=0.001$ (dotted line)
       \end{minipage}
    \end{figure} 

As we can observe on the picture the optimal weights start with value $w^*=0.5$, which is optimal for the fixed effects model for both criteria: $w^*_{D,fix}=w^*_{E,fix}=0.5$, and increase with the variance parameter $\rho$ with limiting values $w^*=0.99$, $w^*=0.98$ and $w^*=0.51$ for the \textit{D}-criterion and $w^*=0.67$, $w^*=0.57$ and $w^*\approx0.5$ for the \textit{E}-criterion for $b=2$, $b=0.6$ and $b=0.001$, respectively. 

On the graphics the values of the optimal designs for the \textit{D}-criterion are larger than those for the \textit{E}-criterion for all $\rho$ and $b$. The \textit{A}-optimal designs illustrated by Figure~1 turn out to take all their values between the corresponding \textit{D}- and \textit{E}-optimal weights. 

Figure 7 and Figure 8 present the efficiency of the optimal designs from the fixed effects model for the present model for the \textit{D}- and \textit{E}- criteria.

\begin{figure}[ht]
    \begin{minipage}[]{8.2 cm}
       \centering
       \includegraphics[width=78mm]{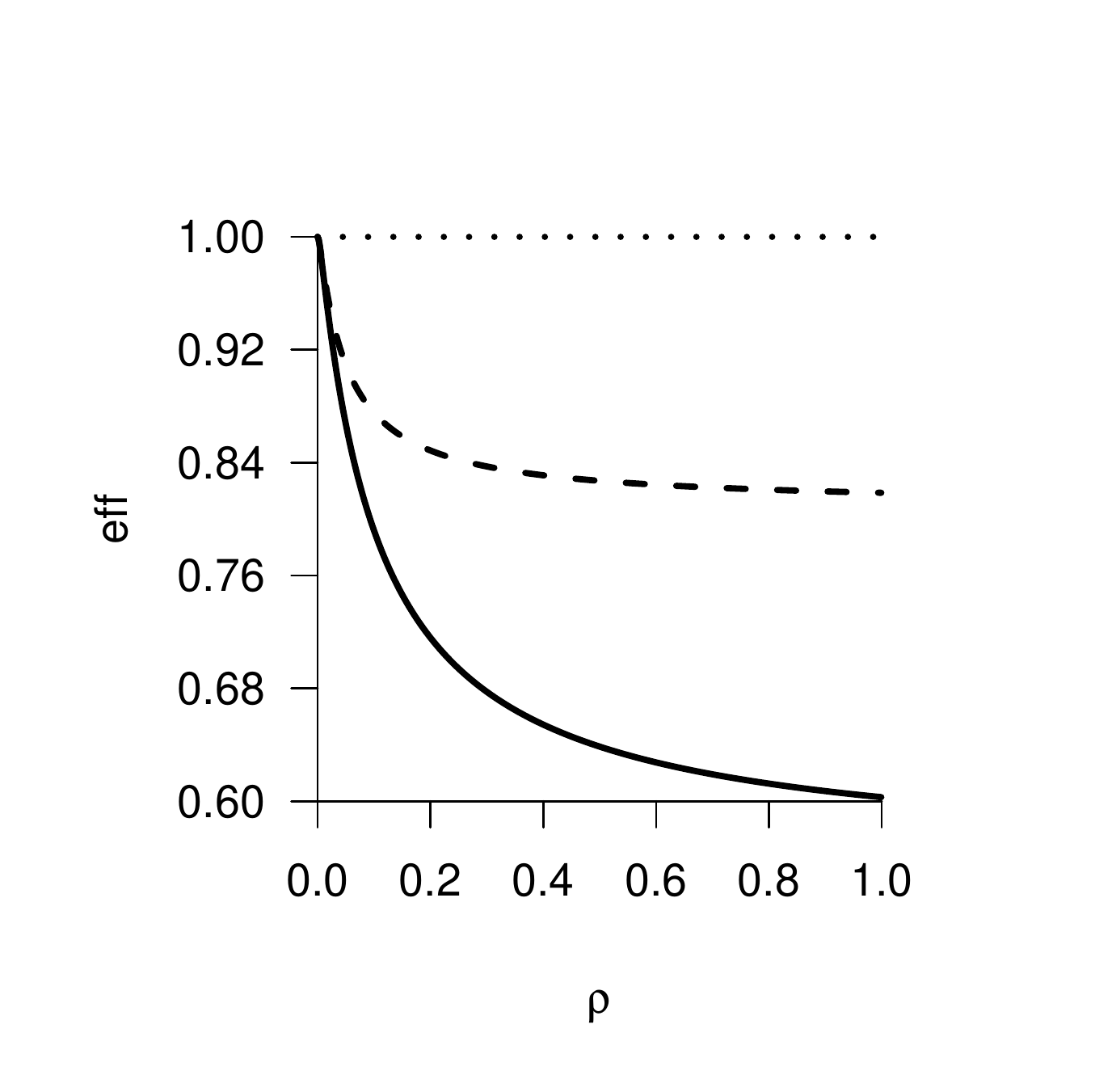}
       \label{d1eff}
       \end{minipage}
       \begin{minipage}[]{8.2 cm}
       \centering
       \includegraphics[width=78mm]{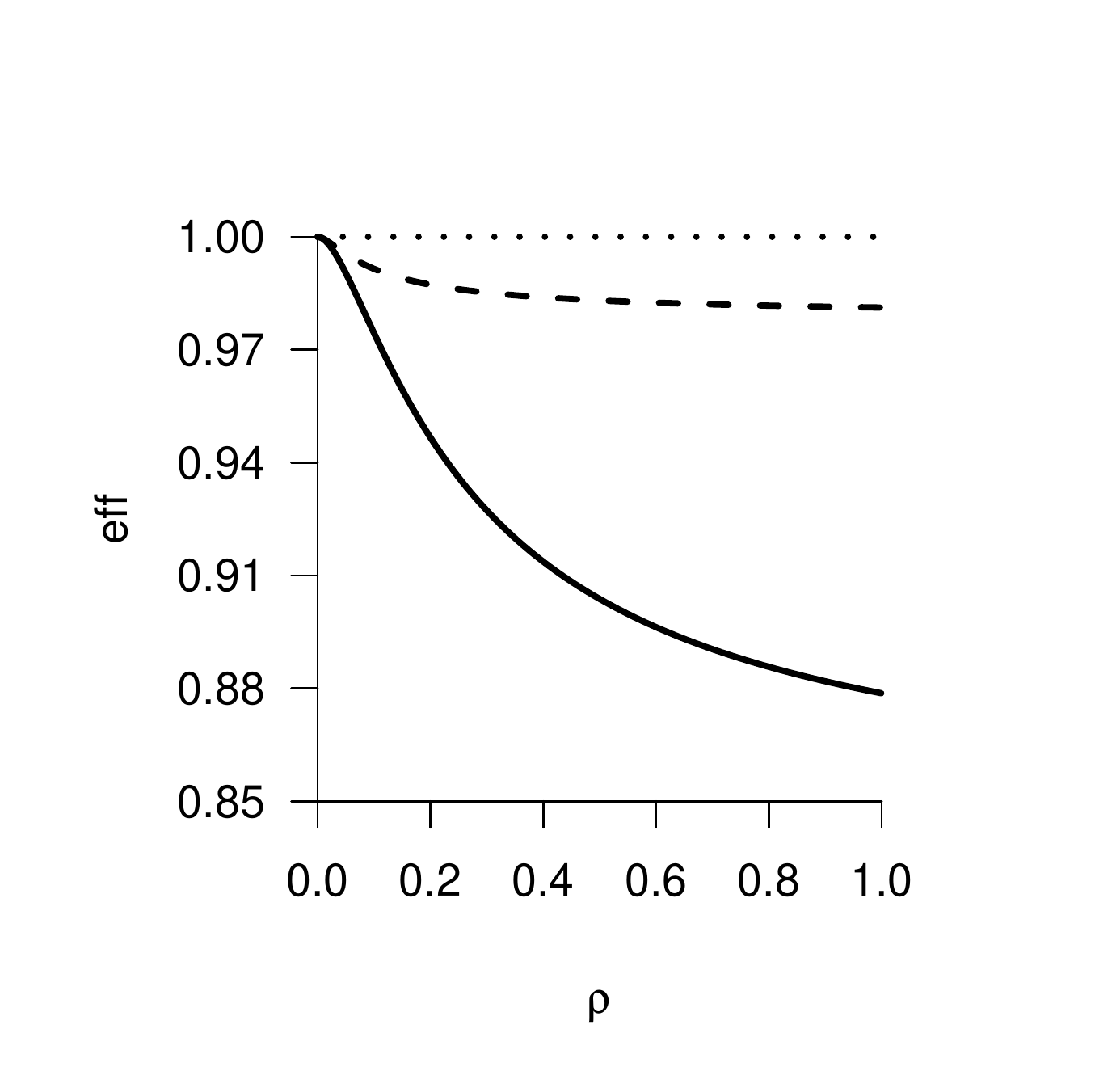}
       \label{e1eff}
       \end{minipage}
       \vspace{-1mm}
       \hspace*{5 mm}
       \begin{minipage}[]{6.5 cm}
       Figure 7: Efficiency of \textit{D}-optimal design in fixed effects model for variance ratios  $b=2$ (solid line), $b=0.6$ (dashed line) and $b=0.001$ (dotted line)
       \end{minipage}
       \hspace{15 mm}
       \begin{minipage}[]{6.5 cm}
       Figure 8: Efficiency of \textit{E}-optimal design in fixed effects model for variance ratios $b=2$ (solid line), $b=0.6$ (dashed line) and $b=0.001$ (dotted line)
       \end{minipage}
    \end{figure} 

The efficiencies on the picture start at point $\textit{eff}=1$ and decrease with increasing variance of the individual treatment effects with limiting values $\textit{eff}=0.60$, $\textit{eff}=0.82$ and $\textit{eff}\approx1$ for the \textit{D}- and $\textit{eff}=0.88$, $\textit{eff}=0.98$ and $\textit{eff}\approx1$ for the \textit{E}-criterion for $b=2$, $b=0.6$ and $b=0.001$, respectively.

According to these graphics and Figure~2 the efficiencies for the \textit{D}-criterion are smaller than those for the \textit{E}-criterion and the \textit{A}-criterion is in between for all values of $\rho$ and $b$.	

Note also that for small values of the variance ratio $b$ the optimal designs are very close to those in the fixed effects model for all design criteria. Hence, the corresponding efficiencies are close to $1$.

\section{Discussion and Conclusions}\label{k5}

In the present work multiple group RCR models with several treatment groups and a control group have been considered. We have obtained \textit{A}-, \textit{D}- and \textit{E}-optimality criteria for the estimation of population parameters and for the prediction of individual treatment effects using the covariance matrix of the BLUE and the mean squared error matrix of the BLUP, respectively. The optimal designs (optimal group sizes) turned out to be different for the estimation and the prediction and do not coincide with those in the corresponding fixed-effects-model (one-way-layout). For large values of the treatment effects variance the optimal designs assign almost all observations to be taken in the treatment groups. If the variance of the individual intercepts is large, the optimal groups sizes tend to those in the fixed effects model. 

The optimal group sizes are locally optimal, i.\,e. they depend on the variance parameters. To avoid this, minimax-optimal designs, which minimize the worst case for the criterion function over some reasonable region of values of the covariance matrix, or some other design criteria, which are robust with respect to the variance parameter, may be considered in the next step of this research. 

For the investigated models we assumed a diagonal covariance matrix of random effects. Models with more complicated covariance structure may also be considered in the future.

%The presented results were obtained using the assumption of approximate designs. However, for convex design criteria the optimal exact designs can be determined by a clever rounding of the optimal approximate designs.

\section*{Acknowledgment} This research has been supported by grant SCHW 531/16-1 of the German Research Foundation (DFG).

\appendix
\section{Appendix}
\subsection{Proofs of Theorems~\ref{t1}-\ref{t4}}\label{A1}
To make use of the available results for estimation and prediction we recognize the model (\ref{12}) as a special case of the linear mixed model (see e.\,g. \cite{chr})
\begin{equation}\label{mod}
\mathbf{Y}=\mathbf{X} \mbox{\boldmath{$\beta $}} + \mathbf{Z} \mbox{\boldmath{$\gamma$}} + \mbox{\boldmath{$\varepsilon$}},
\end{equation}
where $\mathbf{X}$ and\ $\mathbf{Z}$ are known fixed and random effects design matrices, $\mbox{\boldmath{$\beta $}}$ and $\mbox{\boldmath{$\gamma $}}$ are vectors of fixed and random effects, respectively. The random effects $\mbox{\boldmath{$\gamma $}}$ and the observational errors $\mbox{\boldmath{$\varepsilon $}}$ are assumed to be uncorrelated and to have zero means and non-singular covariance matrices $\mathbf{G}=\mbox{Cov}\,(\mbox{\boldmath{$\gamma $}})$ and $\mathbf{R}=\mbox{Cov}\,(\mbox{\boldmath{$\varepsilon $}})$. 

According to \citet{hen1} for full column rank design matrix $\mathbf{X}$ the BLUE $\hat{\mbox{\boldmath{$\beta $}}}$ for $\mbox{\boldmath{$\beta $}}$ and the BLUP $\hat{\mbox{\boldmath{$\gamma $}}}$ for $\mbox{\boldmath{$\gamma $}}$ are provided by the mixed model equations
\begin{equation*}%\label{blup} 
\left( \begin{array}{c} \hat{\mbox{\boldmath{$\beta$}}} \\ \hat{\mbox{\boldmath{$\gamma$}}}
       \end{array} \right)
= {\left( \begin{array}{cc} \mathbf{X}^\top \mathbf{R}^{-1}\mathbf{X} & \mathbf{X}^\top \mathbf{R}^{-1}\mathbf{Z} \\ \mathbf{Z}^\top \mathbf{R}^{-1}\mathbf{X} & \mathbf{Z}^\top \mathbf{R}^{-1}\mathbf{Z}+\mathbf{G}^{-1}
          \end{array} \right)^{-1}}  
\left( \begin{array}{c} \mathbf{X}^\top \mathbf{R}^{-1}\mathbf{Y} \\ \mathbf{Z}^\top \mathbf{R}^{-1}\mathbf{Y} \end{array} \right), 
\end{equation*}
which can be rewritten in the alternative form
\begin{equation}\label{beta} 
\hat{\mbox{\boldmath{$\beta $}}}=\left(\mathbf{X}^\top (\mathbf{Z} \mathbf{G}\mathbf{Z}^\top+\mathbf{R})^{-1}\mathbf{X}\right)^{-1}\mathbf{X}^\top (\mathbf{Z} \mathbf{G}\mathbf{Z}^\top+\mathbf{R})^{-1}\mathbf{Y},
\end{equation}
\begin{equation}\label{gamma}
\hat{\mbox{\boldmath{$\gamma $}}}=\mathbf{G}\mathbf{Z}^\top (\mathbf{Z}\mathbf{G}\mathbf{Z}^\top + \mathbf{R})^{-1}(\mathbf{Y}-\mathbf{X}\hat{\mbox{\boldmath{$\beta $}}}).
\end{equation}

The mean squared error matrix of the estimator and predictor $\left(\hat{\mbox{\boldmath{$\beta $}}}^\top,\, \hat{\mbox{\boldmath{$\gamma $}}}^\top\right)^\top$ is given by (see \cite{hen3})
\begin{equation*}%\label{cov1}
\mbox{Cov}\,\left( \begin{array}{c} \hat{\mbox{\boldmath{$\beta $}}} \\ \hat{\mbox{\boldmath{$\gamma $}}}-\mbox{\boldmath{$\gamma $}} 
                   \end{array} \right)
={\left( \begin{array}{cc} \mathbf{X}^\top \mathbf{R}^{-1}\mathbf{X} & \mathbf{X}^\top \mathbf{R}^{-1}\mathbf{Z} \\ \mathbf{Z}^\top \mathbf{R}^{-1}\mathbf{X} & \mathbf{Z}^\top \mathbf{R}^{-1}\mathbf{Z}+\mathbf{G}^{-1}
          \end{array} \right)^{-1}}
\end{equation*} 
and can be represented as the partitioned matrix
\begin{equation}\label{c}
\mathrm{Cov}\,\left( \begin{array}{c} \hat{\mbox{\boldmath{$\beta $}}} \\ \hat{\mbox{\boldmath{$\gamma$}}}-\mbox{\boldmath{$\gamma $}} \end{array} \right)=\left(
\begin{array}{cc} \mathbf{C}_{11} & \mathbf{C}_{12} \\ \mathbf{C}_{12}^\top  & \mathbf{C}_{22} \end{array} \right),
\end{equation}
where $\mathbf{C}_{11}=\mathrm{Cov}(\hat{\mbox{\boldmath{$\beta $}}})$, $\mathbf{C}_{22}=\mathrm{Cov}\left(\hat{\mbox{\boldmath{$\gamma $}}}-\mbox{\boldmath{$\gamma $}}\right)$,
\begin{equation*}%\label{c11}
\mathbf{C}_{11}=\left(\mathbf{X}^\top \left(\mathbf{Z} \mathbf{G}\mathbf{Z}^\top+\mathbf{R}\right)^{-1}\mathbf{X}\right)^{-1},
\end{equation*}
\begin{equation*}%\label{c22}
\mathbf{C}_{22}=\left(\mathbf{Z}^\top \mathbf{R}^{-1}\mathbf{Z} +\mathbf{G}^{-1}-\mathbf{Z}^\top \mathbf{R}^{-1}\mathbf{X} (\mathbf{X}^\top \mathbf{R}^{-1}\mathbf{X})^{-1}\mathbf{X}^\top \mathbf{R}^{-1}\mathbf{Z}\right)^{-1},
\end{equation*}
\begin{equation*}
\mathbf{C}_{12}= -\mathbf{C}_{11}\, \mathbf{X}^\top \mathbf{R}^{-1}\mathbf{Z}\left(\mathbf{Z}^\top \mathbf{R}^{-1}\mathbf{Z} +\mathbf{G}^{-1}\right)^{-1}.
\end{equation*}

For $\mbox{\boldmath{$\beta $}}=\mbox{\boldmath{$\theta $}}_0$, $\mbox{\boldmath{$\gamma $}}=\mbox{\boldmath{$\zeta $}}$, $\mathbf{X}=\mathrm{Vec}_{j=1}^J\left(\mathds{1}_{r_j}\otimes\left(\mathds{1}_K\,\mathbf{f}(j)^\top\right)\right)$, $\mathbf{Z}=\mathrm{Diag}_{j=1}^J\left(\mathbb{I}_{r_j}\otimes\left(\mathds{1}_K\,\mathbf{f}(j)^\top\right)\right)$,\linebreak $\mathbf{G}=\sigma^2\,\mathbb{I}_N\otimes\textrm{block-diag}(u, \mathbb{I}_{J-1})$ and $\mathbf{R}=\mathrm{Cov}(\mbox{\boldmath{$\varepsilon$}})=\sigma^2\,\mathbb{I}_{NK}$ our model (\ref{12}) is of form (\ref{mod}). 
% Note that the fixed effects design matrix\, $\mathbf{X}$\, has full column rank. Consequently, all linear aspects of the form\, $\mbox{\boldmath{$\Psi$}}=\mathbf{U}\mbox{\boldmath{$\beta$}}+\mathbf{L}\mbox{\boldmath{$\gamma$}}$\, for some\, $h\times J$\, matrix\, $\mathbf{U}$\, and some\, $h\times NJ$\, matrix\, $\mathbf{L}$\, are predictable, i.\,e. there exists a linear unbiased predictor for\, $\mbox{\boldmath{$\Psi$}}$\, (see e.\,g. \cite{pru3}).
%\begin{equation}\label{beta} 
%\hat{\mbox{\boldmath{$\theta $}}}_0=\left(\mathbf{X}^\top (\mathbf{Z} \mathbf{G}\mathbf{Z}^\top+\mathbf{R})^{-1}\mathbf{X}\right)^{-1}\mathbf{X}^\top (\mathbf{Z} \mathbf{G}\mathbf{Z}^\top+\mathbf{R})^{-1}\mathbf{Y},
%\end{equation}
%\begin{equation}\label{gamma}
%\hat{\mbox{\boldmath{$\zeta $}}}=\mathbf{G}\mathbf{Z}^\top (\mathbf{Z}\mathbf{G}\mathbf{Z}^\top + \mathbf{R})^{-1}(\mathbf{Y}-\mathbf{X}\hat{\mbox{\boldmath{$\theta $}}}_0).
%\end{equation}
%Note that the BLUE of\, $\mbox{\boldmath{$\theta$}}_0$\, is unique, while other BLUPs for\, $\mbox{\boldmath{$\zeta$}}$\, (not of the form (\ref{gamma})) can exist.

Using formulas \eqref{beta} and \eqref{gamma} and after employing some linear algebra we obtain the following 
BLUE and BLUP for the fixed and random effects $\mbox{\boldmath{$\theta$}}_0$ and $\mbox{\boldmath{$\zeta$}}$:
\begin{equation*}
\hat{\mbox{\boldmath{$\theta $}}}_0=\left(\begin{array}{c}\bar{\mathbf{Y}}_J \\ \mathrm{Vec}_{j=1}^{J-1}\left(\bar{\mathbf{Y}}_j-\bar{\mathbf{Y}}_J\right)\end{array}\right),
\end{equation*}
and
\begin{equation*}
\hat{\mbox{\boldmath{$\zeta $}}}=\left(\begin{array}{c}\frac{K}{K(v+u)+1}\,\mathrm{Vec}_{j=1}^{J-1}\mathrm{Vec}_{i=N_{J-1}+1}^{N_j}\left(\left(\begin{array}{c}u \\ v\,\mathbf{e}_j\end{array}\right)\left(\bar{\mathbf{Y}}_{j,i}-\bar{\mathbf{Y}}_j\right)\right)\\ \frac{K}{Ku+1}\,\mathrm{Vec}_{i=N_{J-1}+1}^{N_j}\left(\left(\begin{array}{c}u \\ \mathbf{0}_{J-1}\end{array}\right)\left(\bar{\mathbf{Y}}_{J,i}-\bar{\mathbf{Y}}_J\right)\right)\end{array}\right).
\end{equation*}
Now the results \eqref{mu}-\eqref{alpha} of Theorems~\ref{t1} and \ref{t2} are straightforward to verify.
%where\, $\mathbf{e}_j$\, and\, $\mathbf{0}_{J-\textit{1}}$\, are the same as in formula (\ref{reg}).
%According to \cite{hen3}, the mean squared error matrix of the estimation and prediction vector\, $\left(\hat{\mbox{\boldmath{$\theta $}}}_0^\top,\, \hat{\mbox{\boldmath{$\zeta $}}}^\top\right)^\top$ has the following form:
%\begin{equation}\label{c}
%\mathrm{Cov}\,\left( \begin{array}{c} \hat{\mbox{\boldmath{$\theta $}}}_0 \\ \hat{\mbox{\boldmath{$\zeta $}}}-\mbox{\boldmath{$\zeta $}} \end{array} \right)=\left(
%\begin{array}{cc} \mathbf{C}_{11} & \mathbf{C}_{12} \\ \mathbf{C}_{12}^\top  & \mathbf{C}_{22} \end{array} \right),
%\end{equation}
%where the matrices\, $\mathbf{C}_{11}$\,, \,$\mathbf{C}_{12}$\, and\, $\mathbf{C}_{22}$\, are given by
%\begin{equation}\label{c11}
%\mathbf{C}_{11}=\left(\mathbf{X}^\top \left(\mathbf{Z} \mathbf{G}\mathbf{Z}^\top+\mathbf{R}\right)^{-1}\mathbf{X}\right)^{-1},
%\end{equation}
%\begin{equation}\label{c22}
%\mathbf{C}_{22}=\left(\mathbf{Z}^\top \mathbf{R}^{-1}\mathbf{Z} +\mathbf{G}^{-1}-\mathbf{Z}^\top \mathbf{R}^{-1}\mathbf{X} (\mathbf{X}^\top \mathbf{R}^{-1}\mathbf{X})^{-1}\mathbf{X}^\top \mathbf{R}^{-1}\mathbf{Z}\right)^{-1},
%\end{equation}
%\begin{equation}
%\mathbf{C}_{12}= -\mathbf{C}_{11}\, \mathbf{X}^\top \mathbf{R}^{-1}\mathbf{Z}\left(\mathbf{Z}^\top \mathbf{R}^{-1}\mathbf{Z} +\mathbf{G}^{-1}\right)^{-1}
%\end{equation}
%and\, $\mathbf{C}_{11}=\mathrm{Cov}(\hat{\mbox{\boldmath{$\theta $}}}_0)$\,,\, $\mathbf{C}_{22}=\mathrm{Cov}\left(\hat{\mbox{\boldmath{$\zeta $}}}-\mbox{\boldmath{$\zeta $}}\right)$\,.

To proof Theorems~\ref{t3} and \ref{t4} we firstly compute the blocks $\mathbf{C}_{11}$, $\mathbf{C}_{12}$ and $\mathbf{C}_{22}$ of the mean squared error matrix \eqref{c}:
%\begin{eqnarray}\label{C11}
%\mathbf{C}_{11}&=&\sigma^2\left(\frac{Ku+1}{Km}\left(\mathds{1}_{J}\mathds{1}_{J}^\top-2\left(\begin{array}{c} 1 \\ \mathbf{0}_{J-1}\end{array}\right)\left(\begin{array}{c} \mathds{1}_{J-1} \\ 0\end{array}\right)^\top-2\left(\begin{array}{c} \mathds{1}_{J-1} \\ 0\end{array}\right)\left(\begin{array}{c} 1 \\ \mathbf{0}_{J-1}\end{array}\right)^\top\right)\right. \nonumber \\
%&&+\, \left.\frac{K(v+u)+1}{Kn}\left(\mathbf{0}_{J-1}\,\, \vdots \,\, \mathbb{I}_{J-1}\right)^\top\left(\mathbf{0}_{J-1}\,\, \vdots \,\, \mathbb{I}_{J-1}\right)\right),
%\end{eqnarray}
\begin{equation}\label{C11}
\mathbf{C}_{11}=\frac{\sigma^2(Ku+1)}{Km}\left( \begin{array}{cc} 1 & -\mathds{1}_{J-1}^\top \\ -\mathds{1}_{J-1} & \frac{(K(u+v)+1)m}{(Ku+1)n}\mathbb{I}_{J-1}+\mathds{1}_{J-1}\mathds{1}_{J-1}^\top\end{array}\right)
\end{equation}
\begin{eqnarray}\label{C22}
\mathbf{C}_{22}=\sigma^2\left(\begin{array}{cc}\mathbf{C}_{221} & \mathbf{0} \\ \mathbf{0} & \mathbf{C}_{222} \end{array}\right),
\end{eqnarray}
where
\begin{eqnarray*}
\mathbf{C}_{221}&=&\mathbb{I}_{n(J-1)}\otimes\textrm{block-diag}(u, \mathbb{I}_{J-1})\nonumber \\
&&-\,\frac{K}{K(v+u)+1}\mathrm{Diag}_{j=1}^{J-1}\left(\left(\mathbb{I}_n-\frac{1}{n}\mathds{1}_n\mathds{1}_n^\top\right)\otimes\left(\left(\begin{array}{c}u \\ v\,\mathbf{e}_j\end{array}\right)\left(\begin{array}{c}u \\ v\,\mathbf{e}_j\end{array}\right)^\top\right) \right),
\end{eqnarray*}
\begin{eqnarray*}
\mathbf{C}_{222}=\mathbb{I}_{m}\otimes\textrm{block-diag}(u, \mathbb{I}_{J-1})-\frac{k}{Ku+1}\left(\mathbb{I}_m-\frac{1}{m}\mathds{1}_m\mathds{1}_m^\top\right)\otimes\left(\left(\begin{array}{c}u \\ \mathbf{0}_{J-1}\end{array}\right)\left(\begin{array}{c}u \\ \mathbf{0}_{J-1}\end{array}\right)^\top\right),
\end{eqnarray*}
and
\begin{eqnarray}\label{C12}
\mathbf{C}_{12}=-\sigma^2\left(\mathbf{C}_{121}\,\, \vdots \,\, \mathbf{C}_{122}\right),
\end{eqnarray}
where
\begin{eqnarray*}
\mathbf{C}_{121}= \mathrm{tVec}_{j=1}^{J-1}\left(\frac{1}{n}\mathds{1}_n^\top\otimes\left(\left(\begin{array}{c} 0 \\ \mathbf{e}_j\end{array}\right)\left(\begin{array}{c}u \\ v\mathbf{e}_j\end{array}\right)^\top\right)\right),
\end{eqnarray*}
\begin{eqnarray*}%\label{C122}
\mathbf{C}_{122}=\frac{1}{m}\mathds{1}_m^\top\otimes\left(\left(\begin{array}{c} 1 \\ -\mathds{1}_{J-1}\end{array}\right)\left(\begin{array}{c}u \\ \mathbf{0}_{J-1}\end{array}\right)^\top\right).
\end{eqnarray*}

We can observe that $\mbox{\boldmath{$\Psi$}}_0=(\mathbf{0}_{J-1}\, \vdots \, \mathbb{I}_{J-1})\mbox{\boldmath{$\theta$}}_0$ and $\hat{\mbox{\boldmath{$\Psi$}}}_0=(\mathbf{0}_{J-1}\, \vdots \, \mathbb{I}_{J-1})\hat{\mbox{\boldmath{$\theta$}}}_0$ is the BLUE of $\mbox{\boldmath{$\Psi$}}_0$. Then the covariance matrix of $\hat{\mbox{\boldmath{$\Psi$}}}_0$ can be determined using the formula
\begin{equation*}
\mathrm{Cov}\left(\hat{\mbox{\boldmath{$\Psi$}}}_{0}\right)=(\mathbf{0}_{J-1}\, \vdots \, \mathbb{I}_{J-1})\mathbf{C}_{11}(\mathbf{0}_{J-1}\, \vdots \, \mathbb{I}_{J-1})^\top,
\end{equation*}
which implies result \eqref{cov}.

%The vector\, $\mbox{\boldmath{$\Psi$}}_i$\, of individual treatment effects may be represented in the form
%\begin{equation}
%\mbox{\boldmath{$\Psi$}}_i=(\mathbf{0}_{J-1}\, \vdots \, \mathbb{I}_{J-1})\mbox{\boldmath{$\theta$}}_i.
%\end{equation}
For the vector $\mbox{\boldmath{$\Psi$}}$ of all individual treatment effects it cam be verified that
%
%\begin{equation}
%\mbox{\boldmath{$\Psi$}}=\left(\mathbb{I}_N\otimes(\mathbf{0}_{J-1}\, \vdots \, \mathbb{I}_{J-1})\right)\mbox{\boldmath{$\theta$}},
%\end{equation}
\begin{equation*}
\mbox{\boldmath{$\Psi$}}=\left(\mathds{1}_N\otimes(\mathbf{0}_{J-1}\, \vdots \, \mathbb{I}_{J-1})\right)\mbox{\boldmath{$\theta$}}_0+\left(\mathbb{I}_N\otimes(\mathbf{0}_{J-1}\, \vdots \, \mathbb{I}_{J-1})\right)\mbox{\boldmath{$\zeta$}}
\end{equation*}
and the BLUP of $\mbox{\boldmath{$\Psi$}}$ is given by 
\begin{equation*}
\hat{\mbox{\boldmath{$\Psi$}}}=\left(\mathds{1}_N\otimes(\mathbf{0}_{J-1}\, \vdots \, \mathbb{I}_{J-1})\right)\hat{\mbox{\boldmath{$\theta$}}}_0+\left(\mathbb{I}_N\otimes(\mathbf{0}_{J-1}\, \vdots \, \mathbb{I}_{J-1})\right)\hat{\mbox{\boldmath{$\zeta$}}}.
\end{equation*}
Then the mean squared error matrix of $\hat{\mbox{\boldmath{$\Psi$}}}$ is of general form (\ref{mse}) with
\begin{equation*}%\label{b1}
\mathbf{B}_{1}=\left(\mathds{1}_N\otimes(\mathbf{0}_{J-1}\, \vdots \, \mathbb{I}_{J-1})\right)\mathbf{C}_{11}\left(\mathds{1}_N\otimes(\mathbf{0}_{J-1}\, \vdots \, \mathbb{I}_{J-1})\right)^\top,
\end{equation*}
\begin{equation*}
\mathbf{B}_{2}=\left(\mathds{1}_N\otimes(\mathbf{0}_{J-1}\, \vdots \, \mathbb{I}_{J-1})\right)\mathbf{C}_{12}\left(\mathbb{I}_N\otimes(\mathbf{0}_{J-1}\, \vdots \, \mathbb{I}_{J-1})\right)^\top
\end{equation*}
and
\begin{equation*}%\label{b3}
\mathbf{B}_{3}=\left(\mathbb{I}_N\otimes(\mathbf{0}_{J-1}\, \vdots \, \mathbb{I}_{J-1})\right)\mathbf{C}_{22}\left(\mathbb{I}_N\otimes(\mathbf{0}_{J-1}\, \vdots \, \mathbb{I}_{J-1})\right)^\top.
\end{equation*}
After applying \eqref{C11}-\eqref{C12} we obtain the result of Theorem~\ref{t4}.

\subsection{Proof of Lemma~\ref{l2}}\label{A2}
To determine the eigenvalues of $\textrm{Cov}\left(\hat{\mbox{\boldmath{$\Psi$}}}-\mbox{\boldmath{$\Psi$}}\right)$, we have to solve the equation
\begin{equation}\label{eig2}
\textrm{det}\left( \textrm{Cov}\left(\hat{\mbox{\boldmath{$\Psi$}}}-\mbox{\boldmath{$\Psi$}}\right)-\lambda\,\mathbb{I}_{N} \right)=0.
\end{equation}
From \eqref{mse2} it follows that
\begin{equation*}%\label{not}
 \textrm{Cov}\left(\hat{\mbox{\boldmath{$\Psi$}}}-\mbox{\boldmath{$\Psi$}}\right)-\lambda\,\mathbb{I}_{N} =:\left(\begin{array}{cc} \tilde{\mathbf{H}}_{11} & \mathbf{H}_{12} \\ 
\mathbf{H}_{12}^\top & \tilde{\mathbf{H}}_{22} \end{array}\right),
\end{equation*}
where
\begin{equation*}
\tilde{\mathbf{H}}_{11}=(a_1-\lambda)\,\frac{1}{n}\mathds{1}_n\mathds{1}_n^\top+(a_2-\lambda)\left(\mathbb{I}_{n}-\frac{1}{n}\mathds{1}_n\mathds{1}_n^\top\right)
\end{equation*}
for $a_1=\frac{\sigma^2\,N(Ku+1)}{K\,m}$ and $a_2=\frac{\sigma^2\,v\,(Ku+1)}{K(v+u)+1}$,
\begin{equation*}
\tilde{\mathbf{H}}_{22}=a_3\,\frac{1}{m}\mathds{1}_m\mathds{1}_m^\top+(\sigma^2v-\lambda)\mathbb{I}_{m}
\end{equation*}
for $a_3=\frac{\sigma^2(K(v+u)+1)m}{K\,n}+\frac{a_1m}{N}$,
and $\mathbf{H}_{12}$ is the same as in (\ref{mse2}).

Then we compute the determinant of $\textrm{Cov}\left(\hat{\mbox{\boldmath{$\Psi$}}}-\mbox{\boldmath{$\Psi$}}\right)-\lambda\,\mathbb{I}_{N}$ as
\begin{equation*}
\mathrm{det}\left( \textrm{Cov}\left(\hat{\mbox{\boldmath{$\Psi$}}}-\mbox{\boldmath{$\Psi$}}\right)-\lambda\,\mathbb{I}_{N} \right)=\mathrm{det}\left(\tilde{\mathbf{H}}_{11}\right)\mathrm{det}\left(\tilde{\mathbf{H}}_{22}-\mathbf{H}_{12}^\top\,\tilde{\mathbf{H}}_{11}^{-1}\mathbf{H}_{12}\right),
\end{equation*}
where
\begin{equation*}
\mathrm{det}\left(\tilde{\mathbf{H}}_{11}\right)=(a_1-\lambda)(a_2-\lambda)^{n-1},
\end{equation*}
\begin{equation*}
\tilde{\mathbf{H}}_{11}^{-1}=\frac{1}{a_1-\lambda}\,\frac{1}{n}\mathds{1}_n\mathds{1}_n^\top+\frac{1}{a_2-\lambda}\left(\mathbb{I}_{n}-\frac{1}{n}\mathds{1}_n\mathds{1}_n^\top\right),
\end{equation*}
\begin{equation*}
\tilde{\mathbf{H}}_{22}-\mathbf{H}_{12}^\top\,\tilde{\mathbf{H}}_{11}^{-1}\mathbf{H}_{12}=\left(a_3-\frac{a_1^2\,m}{(a_1-\lambda ) n}\right)\frac{1}{m}\mathds{1}_m\mathds{1}_m^\top+\left(\sigma^2v-\lambda\right)\mathbb{I}_{m},
\end{equation*}
\begin{equation*}
\mathrm{det}\left(\tilde{\mathbf{H}}_{22}-\mathbf{H}_{12}^\top\,\tilde{\mathbf{H}}_{11}^{-1}\mathbf{H}_{12}\right)=\left(\sigma^2v-\lambda\right)^{m-1}\left(a_3-\frac{a_1^2\,m}{(a_1-\lambda ) n}+\sigma^2v-\lambda\right).
\end{equation*}
Then we obtain
\begin{equation*}
\mathrm{det}\left( \textrm{Cov}\left(\hat{\mbox{\boldmath{$\Psi$}}}-\mbox{\boldmath{$\Psi$}}\right)-\lambda\,\mathbb{I}_{N} \right)=(a_2-\lambda)^{n-1}\left(\sigma^2v-\lambda\right)^{m-1}\left((a_3+\sigma^2v-\lambda)(a_1-\lambda)-\frac{a_1^2\,m}{n}\right),
\end{equation*}
which results in the following solutions of equation (\ref{eig2}):
\begin{equation*}
\lambda_1=\frac{\sigma^2\,v\,(Ku+1)}{K(v+u)+1},
\end{equation*}
\begin{equation*}
\lambda_2=\sigma^2v,
\end{equation*}
\begin{equation*}
\lambda_3=\frac{\sigma^2N}{2\,Kn\,m}(Km\,v+N(Ku+1)+\sqrt{s_{n,m}}),
\end{equation*}
where $s_{n,m}=K^2m^2v^2+2Km(m-n)(Ku+1)v+N^2(Ku+1)^2$,
and
\begin{equation*}
\lambda_4=\frac{\sigma^2N}{2\,Kn\,m}(Km\,v+N(Ku+1)-\sqrt{s_{n,m}}).
\end{equation*}
%Note that $\lambda_2$ and $\lambda_3$ are independent of the group size, $\lambda_4\geq \lambda_5$ and it can be proved that $\lambda_4\geq \lambda_1$. Then we receive
%\begin{equation*}
%\lambda_{max}\left( \textrm{Cov}\left(\hat{\mbox{\boldmath{$\Psi$}}}-\mbox{\boldmath{$\Psi$}}\right) \right)=\frac{\sigma^2N}{2\,Kn\,m}(Km\,v+N(Ku+1)+\sqrt{s_{n,m}}).
%\end{equation*}
After applying $n=N\,w$ and $m=N(1-w)$ we can see that $s_{n,m}=N^2s_w$ and obtain the results of the lemma.

\bibliographystyle{natbib}
\bibliography{prus}

%% References

%% or manually
% \begin{thebibliography}{1}
% ...
% \end{thebibliography}

\end{document}